\documentclass[a4paper,10pt]{article}

\usepackage{amsmath,amssymb,amsthm}
\usepackage{amssymb,latexsym, bbm,comment,url}

\renewcommand{\theequation}                            
       {\mbox{\arabic{section}.\arabic{equation}}}

{\theoremstyle{plain}
\newtheorem{definition}{Definition}[section]
\newtheorem{lemma}[definition]{Lemma}
\newtheorem{theorem}[definition]{Theorem}

\newtheorem{proposition}[definition]{Proposition}
 
}
{\theoremstyle{definition}

\newtheorem{remark}[definition]{Remark}

}

\renewcommand{\mathbb}{\mathbbm}                     
\renewcommand{\epsilon}{\varepsilon}                 
\renewcommand{\phi}{\varphi}
\renewcommand{\theta}{\vartheta}
\renewcommand{\le}{\leqslant}
\renewcommand{\ge}{\geqslant}

\newcommand{\origfoo}{} \let\origfoo=\sqrt           
\renewcommand{\sqrt}[1]{\origfoo{#1}\;}

\renewcommand{\Re}{\text{\rm Re}\,}                  
\newcommand{\abs}[1]{\left\lvert #1 \right\rvert}    
\newcommand{\norm}[1]{\left\lVert #1 \right\rVert}   

\DeclareMathOperator{\F}{{\cal F}}                   
\DeclareMathOperator{\R}{{\mathbb R}}                
\DeclareMathOperator{\Rp}{{\mathbb R}_+}             
\DeclareMathOperator{\C}{{\mathbb C}}                
\DeclareMathOperator{\N}{{\mathbb N}}                
\DeclareMathOperator{\I}{I}                          %
\newcommand{\A}{{\mathcal A}}
\DeclareMathOperator{\Borel}{{\mathcal B}}
\newcommand{\scapro}[2]{\langle #1,#2\rangle}       
\newcommand{\scaprob}[2]{\big\langle #1,#2\big\rangle}       
\DeclareMathOperator{\1}{\mathbbm 1}
\renewcommand{\S}{{\mathcal S}}

\renewcommand{\SS}{{\L_S}}

\DeclareMathOperator{\M}{{\mathbb M}} 
\renewcommand{\L}{{\mathcal L}}
\renewcommand{\H}{{\mathcal H}}
\DeclareMathOperator{\Z}{{\mathcal Z}}
\renewcommand{\M}{{\mathcal M}}
\newcommand{\E}{{\mathcal E}}
\newcommand{\G}{{\mathcal G}}

\DeclareMathOperator{\Cc}{{\hat{\mathcal Z}}}

\newcommand{\tr}[1]{{\rm tr}\left[ #1 \right]}

\renewcommand\footnotemark{}

\title{Stochastic integration
with respect to\\ cylindrical L{\'e}vy processes}

\author{
Adam Jakubowski\\
Faculty of Mathematics \\  and Computer Science\\
Nicolaus Copernicus University\\
Torun, Poland
\and
Markus Riedle\footnote{The first author acknowledges the Polish NCN grant  2012/07/B/ST1/03508 
and the second author acknowledges the EPSRC grant EP/I036990/1.}\\
Department of Mathematics\\
King's College\\
London WC2R 2LS\\
United Kingdom\\[.2cm] markus.riedle@kcl.ac.uk}

\begin{document}

\maketitle

\begin{abstract}
A cylindrical L{\'e}vy process does not enjoy a cylindrical version of the
semi-martingale decomposition which results in the need to develop a completely novel approach to stochastic integration. In this work, we introduce a stochastic integral for random integrands with respect to cylindrical L{\'e}vy processes in Hilbert spaces. The space of admissible integrands consists of adapted stochastic processes with values in the space of  Hilbert-Schmidt operators. Neither the integrands nor the integrator is required to satisfy any moment or boundedness condition. The integral process is characterised as an adapted, Hilbert space valued semi-martingale with c{\`a}dl{\`a}g trajectories.
\end{abstract}


 \thispagestyle{empty}

\section{Introduction}

Cylindrical Brownian motion is the most prominent model of the driving noise for stochastic partial differential equations. The attribute {\em cylindrical} refers here to the fact that cylindrical Brownian motions are not classical stochastic processes attaining values in the underlying space but are generalised objects whose probabilistic distributions are described by a cylindrical, i.e.\ a finitely additive, measure. The reasons for the choice of cylindrical but not classical Brownian motion can be found in the facts that there does not exist a classical Brownian motion with independent components, i.e.\ a standard Brownian motion, in an infinite dimensional Hilbert space, and that cylindrical processes enable a very flexible modelling of random noise in time and space.

The concept of cylindrical Browian motion is naturally extended to {\em cylindrical L{\'e}vy processes} in one of the authors' work \cite{ApplebaumRiedle} with Applebaum. Some specific examples and their constructions of cylindrical L{\'e}vy processes are presented in the work \cite{Riedle-Cauchy} by Riedle.
Linear and semi-linear stochastic partial differential equations perturbed by an additive noise which is  modelled by
various but specific examples of cylindrical L{\'e}vy processes can be found for example in the works
  Brze\'zniak and Zabczyk \cite{BrzZab10},
 Peszat and Zabczyk \cite{PeszatZab12}, and Priola and Zabczyk \cite{PriolaZabczyk}.
However, modelling an arbitrary perturbation of a general stochastic partial differential equations beyond the purely additive case requires a theory of stochastic integration of random integrands with respect to cylindrical L{\'e}vy processes.

Stochastic integration with respect to  cylindrical Brownian processes is developed for example in Daletskij \cite{Daletskij}, followed by the articles Gaveau \cite{Gaveau}, Lepingle and Ouvrard \cite{LepingleOuvrard} and many others.  Surprisingly, stochastic integration with respect to other cylindrical processes than cylindrical Brownian motion is much less considered.  In fact, only  with respect to cylindrical martingales a stochastic integration theory is developed which originates either from an approach by M{\'e}tivier  and Pellaumail in  \cite{MetivierPellcylindrical} and \cite{MetivierPell} or from Mikulevi\v{c}ius and Rozovski\v{\i} in \cite{MikRoz98} and \cite{MikRoz99}.  The construction by  M{\'e}tivier  and Pellaumail is based on  Dol\'eans measures whereas the construction by Mikulevi\v{c}ius and Rozovski\v{\i} uses  a family of reproducing kernel Hilbert spaces. Thus, both constructions heavily rely  on the assumed existence of finite weak second moments. In M{\'e}tivier  and Pellaumail \cite{MetivierPellcylindrical}, the construction is extended to cylindrical local martingales.
For the special case of a cylindrical L{\'e}vy process with finite weak second moments one can follow a classical It{\^o} aproach to define the stochastic integral for random integrands; see
Riedle \cite{Riedle14}.

To our best knowledge, each approach to stochastic integration with respect to classical L{\'e}vy processes or  classical semi-mar\-tin\-gales
is based on the semi-martingale decomposition of the integrator into a local martingale and a process of bounded variation.
However, this approach fails for cylindrical L\'evy process although they are in the class
of {\em cylindrical semi-martingales}. This is due to the  conceptual mismatch that a cylindrical semi-martingale cannot be decomposed into the sum of a cylindrical local martingale and another cylindrical process, see
Remark \ref{re.semi-martingale}.
Consequently, our work requires a novel approach to stochastic integration  without decomposing the integrator. Even in the finite dimensional case we are not aware of such a kind of approach. In our setting it is even more intricate due to the infinite dimensionality and the generalised process not attaining values in the underlying space.

To explain our  approach in more detail let $(Y(t):\, t\in [0,T])$ be a classical L{\'e}vy process in a Hilbert space $U$ with inner product $\scapro{\cdot}{\cdot}$. A simple integrand $(\Psi(t):\, t\in [0,T])$ is of the form $\Psi=\1_{(a,b]}\otimes\, \Phi$  where $0\le a\le b\le T$ and $\Phi $ is a random variable with values in the space of Hilbert-Schmidt operators from $U$ to another Hilbert space $V$. Each sensible
definition of stochastic integration leads to
\begin{align}\label{eq.intro-int-Y}
\Big\langle \int_0^T \Psi(s)\, dY(s)\Big\rangle \Big\langle v\Big\rangle
=\big\langle \Phi\big( Y(b)-Y(a)\big)\big\rangle \big\langle v\big\rangle
=\big\langle  Y(b)-Y(a)\big\rangle \big\langle \Phi^\ast v\big\rangle
\end{align}
 for every $v\in V$, where $\Phi^\ast$ denotes the adjoint operator.
A cylindrical process, like the cylindrical L{\'e}vy process, is a family $(L(t)\colon \, t\in [0,T])$ of linear and bounded operators $L(t)$ from $U$ to the space  of equivalence classes of real valued random variables.
If we substitute $Y$ by the cylindrical L{\'e}vy process $L$ in \eqref{eq.intro-int-Y} then the inner product on the right hand side in \eqref{eq.intro-int-Y}
corresponds to the application of the linear operator $L(b)-L(a)$  to  the
other argument of the inner product such that we arrive at:
\begin{align}\label{eq.intro-int-L}
\Big\langle \int_0^T \Psi(s)\, dL(s)\Big\rangle \Big\langle v\Big\rangle
=\big( L(b)-L(a)\big)\big(\Phi^\ast v\big).
\end{align}
 However, a technical and a conceptual problem  arise in \eqref{eq.intro-int-L}:
\begin{enumerate}
\item[(1)] the linear operator $L(b)-L(a)$, mapping to the space of equivalence classes of
random variables, is applied to a random argument which results in an ambiguity;
\item[(2)] in order to obtain a $V$-valued stochastic integral such as
$\Phi\big( Y(b)-Y(a)\big)$ in \eqref{eq.intro-int-Y}, there must
exist a $V$-valued random variable $J$ satisfying
 \begin{align*}
\big( L(b)-L(a)\big)\big(\Phi^\ast v\big)
=\big\langle  J\big\rangle \big\langle  v\big\rangle
 \qquad\text{for all }v\in V.
 \end{align*}
\end{enumerate}
We call the approach for solving the Problems (1) and (2) the {\em radonification of the increments} and present it in Section \ref{se.increments}.

However, a much more complicated problem is to extend the class of admissible integrands to a larger space rather than only simple integrands. Denote by  $\H_0(U,V)$  the space of linear combinations of simple integrands of the form as $\Psi$ above. Then,
by means of the radonification of the increments one   can define an integral operator
\begin{align}\label{eq.int-op}
 I\colon \H_0(U,V)\to L_P^0(\Omega;V),
\end{align}
where $L_P^0(\Omega;V)$ denotes the space of equivalence classes of $V$-valued random variables. In the classical setting, the integrator $Y$ is decomposed into a martingale $M$ and a bounded variation process $A$, resulting in integral operators $I_A$ and $I_M$ with $I=I_A+I_M$. It is straightforward to extend the domain of the integral operator $I_A$. The integral operator
$I_M$ turns out to map to the Hilbert space $ L_P^2(\Omega;V)$ of equivalence classes of $V$-valued random variables with finite second moments. Martingale properties and the {\em nice} Hilbert space topology of $ L_P^2(\Omega;V)$ allow to conclude the continuity of $I_M$ and thus to extend  its domain.
However, as mentioned above,  the cylindrical L{\'e}vy process $L$  does not enjoy an analogue
decomposition, and thus we
must work with the integrator operator  \eqref{eq.int-op} {\em in a single entity} to solve:
\begin{enumerate}
\item[(3)] if a sequence $(\Psi_n)_{n\in\N}$ of simple processes in $\H_0(U,V)$ converges to
a stochastic process $\Psi$ in a larger space in some
sense then $I(\Psi_n)$ converges to a random variable in $L_P^0(\Omega;V)$.
\end{enumerate}
Dealing with problem (3) means in particular  that, instead of exploiting the It{\^o} isomorphism
to the Hilbert space $ L_P^2(\Omega;V)$, one must establish convergence in the  much less amenable topology in $L_P^0(\Omega;V)$, i.e.\ convergence in probability. We solve Problem (3) in Section \ref{se.stochastic-integral}. Here, the main
step is establishing  tightness of the set $\{I(\Psi_n):\, n\in\N\}$ of Hilbert space valued
random variables.
We prove tightness of this set by exploiting the result
that tightness of the sum of the {\em decoupled tangent sequence} implies tightness of the original sum. This result originates from one of the authors' work \cite{JakubowskiPhD}, and we will introduce and prove a modified version of this result in Section \ref{se.tightness}. Although this result was originally introduced with a completely different aim it seems to be tailor-made for
the considerations of our current work.

\section{Preliminaries}\label{se.preliminaries}

Let $U$ and $V$ be separable Hilbert spaces with inner products $\scapro{\cdot}{\cdot}$ and corresponding norms $\norm{\cdot}$. The dual spaces are identified by the original Hilbert spaces.  The unit ball is denoted by $B_V:=\{v\in V:\,\norm{v}\le 1\}$.
 Throughout the paper, $\{e_k\}_{k\in\N}$ and $\{f_k\}_{k\in\N}$ denote some orthonormal basis of $U$ and $V$, respectively.

The space of linear and bounded operators is denoted by $\L(U,V)$ and it is equipped with the operator norm $\norm{\cdot}_{U\to V}$. The subspace of {\em Hilbert-Schmidt operators} is denoted by $\L_2(U,V)$ and  is equipped with the norm
\begin{align*}
\norm{\phi}_{\L_2}^2:=\sum_{k=1}^\infty \norm{\phi e_k}^2.
\end{align*}
A simple argument using the standard characterisation of compact sets in Hilbert
spaces show that a set $K\subseteq\L_2(U,V)$ is compact if and only if it is bounded, closed
and obeys
\begin{align}
  \lim_{N\to\infty} \sup_{\phi\in K}\sum_{k=N+1}^\infty \norm{\phi e_k}^2=0.
   \label{eq.HS-compact2}
\end{align}
The space of $\L_2(U,V)$-valued c\`agl\`ad (continue \`a  gauche, limite \`a  droite) functions is denoted by
\begin{align*}
D_-\big([0,T];\L_2(U,V)\big)
:=\Big\{\psi\colon [0,T]\to \L_2(U,V): \text{left-continuous with right-limits}\Big \}.
\end{align*}
This space becomes complete under the Skorokhod metric
\begin{align}\label{de.Skorokhod-metric}
  d_{J}(\phi,\psi):=\inf_{j\in \Lambda} \Big(\sup_{t\in [0,T]}\norm{\phi(t)-\psi\circ j(t)}_{\L_2} \vee \sup_{t\in [0,T]}\abs{t-j(t)}\Big),
\end{align}
where the infimum is over the set $\Lambda$ of all all strictly increasing, continuous bijections $j\colon [0,T]\to [0,T]$.

The Borel $\sigma$-algebra in $U$ is dented by $\Borel(U)$ and the space of Borel measures on $\Borel(U)$ is denoted by $\M(U)$. The space of Borel probability measures
is denoted by $\M_1(U)$ and it is equipped  with the Prokhorov metric
\begin{align*}
d_P(\mu,\nu):=
\inf\big\{\epsilon>0:\mu(B)\le \nu(B_\epsilon)+\epsilon
\text{ and } \nu(B)\le \mu(B_\epsilon)+\epsilon
\text{ for closed $B\in \Borel(U)$}\big\},
\end{align*}
where $B_\epsilon:=\{u\in U:\, \inf\{\norm{u-b}: b\in B\} < \epsilon\}$.
Convergence in the Prokhorov metric is equivalent to weak convergence of probability measures.

Let $(\Omega,\A,P)$ be a probability space. The space of equivalence classes of measurable functions $X\colon \Omega\to U$ is denoted by $L_P^0(\Omega;U)$. If $\G$ is a sub-$\sigma$-algebra of $\A$ we write
$L_P^0(\Omega,\G;U)$ for the space of equivalence classes of $\G$-measurable functions. By defining the function
\begin{align}\label{eq.def-metric-L0}
p\colon L_P^0(\Omega;U)\to [0,1],\qquad
 p(X)= E\Big[1\wedge \norm{X}^2\Big],
\end{align}
the space $L_P^0(\Omega;U)$ becomes an $F$-space under the  metric
$d(X,Y):=p(X-Y)$.

Let $S$ be a subset of $U$. For every elements $u_1,\dots, u_n\in S$, $n\in\N$ and $B\in\Borel(\R^n)$ define
\begin{align*}
  C(u_1,\dots, u_n;B):=\big\{u\in U:\big( \scapro{u}{u_1},\dots, \scapro{u}{u_n}\big)\in B\big\}.
\end{align*}
These sets are called {\em cylindrical sets with respect to $S$} and they form an algebra $\Z(U,S)$. The generated $\sigma$-algebra is denoted by $\Cc(U,S)$ and it is called the {\em cylindrical $\sigma$-algebra with
respect to $S$}. If $S=U$ we write $\Z(U):=\Z(U,S)$ and
$\Cc(U):=\Cc(U,S)$.

A function $\eta\colon \Z(U)\to [0,\infty]$ is called a {\em cylindrical measure on $\Z(U)$} if for each finite subset $S\subseteq U$ the restriction of $\eta$ to the $\sigma$-algebra $\Cc(U,S)$ is a measure. A cylindrical
measure $\eta$ is called finite if $\eta(U)<\infty$ and a cylindrical probability
measure if $\eta(U)=1$. The characteristic function
of a finite cylindrical measure $\eta$ is defined by
\begin{align*}
\chi_\eta\colon U\to \C,\qquad \chi_\eta (u)=\int_U e^{i\scapro{u}{h}}\, \eta(dh).
\end{align*}
Note that this integral is well defined as the integrand is measurable with respect to $\Cc(U,\{u\})$ for each $u\in\ U$.
The cylindrical measure $\eta$ is called {\em continuous} if $\chi_\eta$ is continuous.

A {\em cylindrical random variable in $U$} is a linear and continuous mapping
\begin{align*}
 Z\colon U\to L_P^0(\Omega;\R).
\end{align*}
If $C=C(u_1,\dots, u_n;B)$ is a cylindrical set for
$u_1,\dots, u_n\in U$ and $B\in \Borel(\R^n)$ we obtain a cylindrical probability measure $\eta$ by the definition
\begin{align*}
  \eta(C):=P\big((Zu_1,\dots, Zu_n)\in B\big).
\end{align*}
The mapping $\eta$ is called the {\em cylindrical distribution of $Z$}. The characteristic function of a cylindrical random variable $Z$
is defined by
\begin{align*}
\phi_Z\colon U\to \C, \qquad \phi_Z(u)= E\big[ \exp(iZu)\big],
\end{align*}
and the characteristic function of $Z$ and its  cylindrical distribution $\eta$ coincide.
For a function $\phi\in \L(U,V)$ one can define a cylindrical random variable in $V$ by
\begin{align*}
  Z_\phi\colon V\to L_P^0(\Omega;\R),\qquad
  Z_\phi v= Z(\phi^\ast v).
\end{align*}
In general, $Z_\phi$ is only a cylindrical random variable but if $\phi$ is a Hilbert-Schmidt operator  then there exists a $V$-valued random variable $\phi(Z)\colon\Omega\to V$ satisfying
\begin{align}\label{eq.HSradonifying}
  Z(\phi^\ast v)=\scapro{\phi(Z)}{v}\qquad\text{for all }v\in V;
\end{align}
see \cite[Th.VI.5.2]{Vaketal}.

A family $(Z(t):\,t\ge 0)$ of cylindrical random variables $Z(t)$ in $U$ is called
a {\em cylindrical process in $U$}. In our work \cite{ApplebaumRiedle}, we extended the concept of
cylindrical Brownian motion to cylindrical L{\'e}vy processes:
\begin{definition}
A cylindrical process $(L(t):\, t\ge 0)$ in $U$ is called a {\em cylindrical L{\'e}vy process} if for each $n\in\N$ and any $u_1,\dots, u_n\in U$ we have that
\begin{align*}
\big( (L(t)u_1,\dots, L(t)u_n):\, t\ge 0\big)
\end{align*}
is a L{\'e}vy process in $\R^n$.
\end{definition}
The characteristic function of $L(t)$ is studied in detail in our work \cite{Riedle11}. It turns out that
the characteristic function of $L(t)$ for each $t\ge 0$ is of the form
\begin{align*}
\phi_{L(t)}\colon U \to\C,\qquad  \phi_{L(t)}(u)&=\exp\big(t S(u)\big),
\end{align*}
where $S\colon U \to \C$ is called the {\em cylindrical symbol of $L$} and  is of the form
\begin{align}\label{eq.Levy-symbol}
S(u)=i a(u) -\tfrac{1}{2} \scapro{Qu}{u}
   +\int_U\left(e^{i\scapro{u}{h}}-1- i\scapro{u}{h}   \1_{B_{\R}}(\scapro{u}{h})\right)\nu(dh).
\end{align}
Here, $a\colon U\to\R$ is a continuous mapping with $a(0)=0$,  $Q\colon U \to \R$ is a positive and symmetric operator and $\nu$ is a cylindrical measure on $\Z(U)$ satisfying
\begin{align*}
  \int_U \big(\scapro{u}{h}^2 \wedge 1\Big) \,\nu(dh)<\infty
  \qquad\text{for all }u\in U.
\end{align*}
Since $L(t)\colon U\to L_P^0(\Omega;\R)$ is continuous, it follows that the characteristic function
$\phi_{L(1)}\colon U\to \C$ is continuous, and thus the symbol $S\colon U\to \C$ is continuous. According to Lemma 3.2 in \cite{Riedle-Cauchy} the cylindrical symbol $S$ maps bounded sets to bounded sets.

\begin{remark}\label{re.semi-martingale}
It follows from the L\'evy-It{\^o} decomposition in $\R$ that for each $u\in U$ and $t\ge 0$ a cylindrical L\'evy process  $L$ with L\'evy symbol \eqref{eq.Levy-symbol} can be decomposed into
\begin{align*}
  L(t)u=a(u)t+ W(t)u + \int_{\abs{\beta}\le 1} \beta \tilde{N}_u(t,d\beta) +
  \int_{\abs{\beta}> 1} \beta N_u(t,d\beta),
\end{align*}
where $W$ is a cylindrical Wiener process in $U$
with covariance operator $Q$ and
\begin{align*}
  N_u(t,B):=\sum_{0\le s\le t} \1_B\big((L(s)u-L(s-))u\big)
  \qquad\text{for all }t\in [0,T],\, B\in\Borel(\R\setminus\{0\}),
\end{align*}
and $\tilde{N}_u$ is the compensated Poisson random measure defined by
$\tilde{N}_u(t,B):=N_u(t,B)-t(\nu\circ \scapro{\cdot}{u}^{-1})(B)$; see \cite[Th.3.9]{ApplebaumRiedle}.
If $L$ does not have finite weak second moments, that is $E\big[\abs{L(1)u}^2\big]=\infty$, then the
(local) martingale part in the semi-martingale decomposition of $(L(t)u:\, t\in [0,T])$ is given by
the sum $W(t)u+R(t)(u)$ where
\begin{align*}
  R(t)(u):=\int_{\abs{\beta}\le 1} \beta \tilde{N}_u(t,d\beta).
\end{align*}
As the truncation function $\beta\mapsto \1_{B}(\beta)$ is not linear the mapping $u\mapsto R(t)(u)$
is not linear neither. Thus, $(L(t)u:\, t\in [0,T])$ enjoys a semi-martingale decomposition for fixed $u\in U$,
but the martingale and bounded variation parts are not linear in $u$ in general.
\end{remark}

We equip the probability space $(\Omega,\A,P)$ with the filtration generated by $L$ and defined by
\begin{align*}
  \F_t:=\sigma(\{L(s)u:\, u\in U, \, s\in [0,t]\})\qquad\text{for all }t\ge 0.
\end{align*}
For a filtration  $\G:=\{\G_{t}\}_{t\in I}$  where $I\subseteq [0,\infty)$ is
an arbitrary index set we define
\begin{align*}
\Upsilon(\G):=\big\{\tau\colon \Omega\to I:\, \text{is stopping time for $\G$}\big\}.
\end{align*}

\section{Tightness by decoupling}\label{se.tightness}

In the later part of this work, the main argument on extending the definition of the stochastic integral from simple integrands to a much larger class of integrands is based on establishing tightness of the set of stochastic  integrals for a sequence of simple integrands. This will be established by the following result which provides a handy criterion for the tightness of a set of sums of random variables in a Hilbert space. The theorem is a modification of a result by Jakubowski in \cite{JakubowskiPhD}, and which is also published in a more general setting in \cite{Jakubowski88}.

\begin{theorem}\label{th.tightness-tangent}
For each $n\in\N$ let $\{X_{n,k}:\,k\in\N\}$ be  a sequence of $V$-valued random variables
adapted to a filtration $F_n:=\{\F_{n,k}:\,k\in\N_0\}$. Define for each $k$, $n\in\N$
a version of the regular conditional distribution
\begin{align*}
  P_{n,k}\colon\Borel(V)\times\Omega\to [0,1], \quad P_{n,k}(B,\omega)=P\big(X_{n,k}\in B\mid \F_{n,k-1}\big)(\omega).
\end{align*}
If there exists a sequence $\{\sigma_n:\, n\in\N\}$ of finite stopping times $\sigma_n\in \Upsilon(F_n)$
such that
$  \{P_{n,1}\ast \dots \ast P_{n,\tau}:\,\tau\in\Upsilon(F_n),\, 1\le \tau\le \sigma_n,\,n\in\N \} \quad\text{is tight,}$
then
\begin{align}\label{eq.random-field-tight}
  \{X_{n,1}+\dots + X_{n,\tau}:\, \tau\in\Upsilon(F_n),\, 1\le \tau\le \sigma_n,\,n\in\N  \}
\end{align}
is tight.
\end{theorem}

Theorem  \ref{th.tightness-tangent} provides a method for establishing tightness
of the random field \eqref{eq.random-field-tight}. We call this method {\em tightness by decoupling} for the following reason: for the given sequences $\{X_{n,k}:\,k\in\N\}$ there exist sequences $\{X^\ast_{n,k}:\,k\in\N\}$ of random variables $X^\ast_{n,k}$ on a larger probability space $(\Omega^\ast,\A^\ast, P^\ast)$ and a $\sigma$-algebra $\G\subseteq \A^\ast$ satisfying:
\begin{enumerate}
\item[(i)] for every $n\in\N$ the sequence
$\{X_{n,k}^\ast:\,k\in\N\}$ is conditionally independent given $\G$;
\item[(ii)] $P^\ast(X_{n,k}^\ast\in B|\G)=
P(X_{n,k}\in B|\F_{n,{k-1}})$  for all $B\in\Borel(V)$ and $k$, $n\in\N$.
\end{enumerate}
The sequence $\{X^\ast_{n,k}:\,k\in\N\}$ is called the
{\em decoupled tangent sequence}; see Chapter 6 in \cite{dePenaGine} or \cite{Kwapien}. By defining
\begin{align*}
 S_n(\sigma_n):=
\begin{cases}
0 &\text{if } \sigma_n=0, \\\displaystyle
 \sum_{k=1}^{\sigma_n} X_{n,k}  &\text{else,}
\end{cases}
\qquad\qquad
 S_n^\ast(\sigma_n):=
\begin{cases}
0 &\text{if } \sigma_n=0, \\\displaystyle
 \sum_{k=1}^{\sigma_n} X_{n,k}^\ast  &\text{else,}
\end{cases}
 \end{align*}
one can conclude from Theorem \ref{th.tightness-tangent} that
if $\{S_n^\ast(\sigma_n):\,n\in\N\}$ is tight then
$\{S_n(\sigma_n):\,n\in\N\}$ is also tight.

Applying Theorem \ref{th.tightness-tangent} in the one-dimensional case yields another result, the {\em principle of conditioning}, which we also use in this work. The original proof can be found in
\cite{Beskaetal} and \cite{Jakubowski80}, and further extensions to Hilbert spaces  in \cite{Jakubowski86}.
\begin{theorem}\label{th.conditioning}
For each $n\in\N$ let $\{X_{n,k}:\,k\in\N\}$ be  a sequence of real valued random variables
adapted to a filtration $F_n:=\{\F_{n,k}:\,k\in\N_0\}$ and $\sigma_n\colon\Omega\to \N$ be a stopping time for $\{\F_{n,k}:\,k\in\N\}$.
 Define for each $k$, $n\in\N$:
\begin{align*}
  \Delta_{n,k}\colon\R\times\Omega\to \C, \qquad \Delta_{n,k}(\beta,\omega)=E\big[e^{i\beta X_{n,k}}\mid \F_{n,k-1}\big](\omega).
\end{align*}
If for each $\beta\in\R$ there exists a deterministic constant $c(\beta)\neq 0$ such that
\begin{align*}
   \lim_{n\to\infty} \prod_{k=1}^{\sigma_n}\Delta_{n,k}(\beta,\cdot)= c(\beta)
   \quad\text{in probability},
\end{align*}
then it follows that
\begin{align*}
  \lim_{n\to\infty} E\left[e^{i\beta (X_{n,1}+\dots + X_{n,\sigma_n})}\right]= c(\beta).
\end{align*}
\end{theorem}

For the proof of Theorem \ref{th.conditioning} we refer to the literature. For the following proof of Theorem \ref{th.tightness-tangent} we introduce a few notations and objects. A {\em random measure} is a measurable mapping $M:\Omega\to \M(V)$, where measurability is with respect to the Borel $\sigma$-algebra induced by the Prokhorov metric in $\M(V)$. The mapping $M$ is called a {\em random probability measure}, if it maps to the space $\M_1(V)$ of Borel probability measures on $\Borel(V)$.
A random measure $M$ is called {\em integrable}, if $E[M(V)]<\infty$. In this case, $E[M](B):=E[M(B)]$ for all $B\in \Borel(V)$ defines an element in $\M(V)$.
By starting with simple functions and passing to the limit one shows for bounded, measurable functions $\phi\colon V\to \R$ that
\begin{align}\label{eq.fubini-randomE-general}
E\left[\int_V \phi(u)\, M(du)\right]
=\int_V \phi(u) \,E[M](du).
\end{align}
The characteristic function of a random probability measure $M$ is defined by
\begin{align*}
  \chi_M\colon V\times\Omega\to\C, \qquad \chi_M(v)=\int_V e^{i\scapro{v}{h}}\,M(dh).
\end{align*}
For an integrable random meaure $M$ it follows from \eqref{eq.fubini-randomE-general} that
\begin{align}\label{eq.fubini-randomE}
E[\chi_M(v)]=\chi_{E[M]}(v)\qquad \text{for all $v\in V$.}
\end{align}
A set $\{M_i:\, i\in I\}$ of random measures  is called {\em tight}, if for each $\epsilon_1>0$ we have:
\begin{align}
\begin{split}\label{eq.def-tight-rm}
& \text{there exist for all $i\in I$ a set $A_i\in\A $ with $P(A_i)>1-\epsilon_1$ obeying:}\\
& \qquad \big\{M_i(\cdot,\omega):\, \omega\in A_i,\, i\in I\big\} \quad\text{is relatively compact in $\M(V)$.}
\end{split}
\intertext{It follows from Prokhorov's theorem that Condition \eqref{eq.def-tight-rm} is equivalent to}
& \text{\rm (i) }  \sup_{i\in I}\sup_{w\in A_i} M_i(V,\omega)<\infty,\label{eq.rm-tight-bounded} \\
& \text{\rm (ii) }\label{eq.rm-tight-tight}
 \text{for each $\epsilon_2>0$ there exists a compact set $K\subseteq V$ (depending on $\epsilon_1$) such that}\notag\\
& \hspace*{3cm}   \sup_{i\in I}\sup_{w\in A_i}M_i(K^c,\omega)\le \epsilon_2.
\end{align}
If $\{M_i:\, i\in I\}$ is a family of random probability measures then it is tight if and only if $\{E[M_i]:\, i\in I\}$ is tight in $\M(V)$.

A non-negative, symmetric operator $\phi\colon V\to V$ is called an {\em $S$-operator} if
\begin{align*}
  \tr{\phi}:=\sum_{k=1}^\infty \scapro{\phi f_k}{f_k}<\infty .
\end{align*}
The space of all $S$-operators is denoted by $\SS(V)$. The space $\SS(V)$ is a subspace of the Banach space of trace class operators, and it is equipped with the Borel $\sigma$-algebra of the latter.  A set $\{\phi_i:\, i\in I\}\subseteq \SS(V)$ is relatively compact if and only if
\begin{align}
&\text{\rm (i) }  \sup_{i\in I} \tr{\phi_i}<\infty;\label{eq.compactS-finite-trace}\\
& \text{\rm (ii) } \lim_{N\to\infty} \sup_{i\in I}\sum_{k=N}^\infty \scapro{\phi_if_k}{f_k}=0.
\hspace*{3cm} \label{eq.compactS-sup-limit}
\end{align}
A set $\{T_i:\, i\in I\}$ of random variables $T_i\colon \Omega\to\SS(V)$ is {\em tight} if and only if for each $\epsilon>0$ there exist for all $i\in I$  a set $A_i\in \A$ with  $P(A_i)>1-\epsilon$ such that
\begin{align}\label{eq.S-tight}
 \{T_i(\omega):\, \omega\in A_i,\,i\in I\}\quad\text{is relatively compact in $\SS(V)$}.
\end{align}

\begin{proof} of Theorem \ref{th.tightness-tangent}.
For each $\tau\in \Upsilon(F_n)$ with $\tau\ge 1$ and $n\in\N$ define
the random probability  measure
\begin{align*}
 P_n(\tau)\colon\Borel(V)\times\Omega\to [0,1],\qquad P_n(\tau):=P_{n,1}\ast\dots \ast P_{n,\tau},
\end{align*}
and, by denoting $S_n(\tau):=X_{n,1}+\dots +X_{n,\tau}$, the random probability measure
\begin{align*}
  Q_n(\tau)\colon\Borel(V)\times\Omega\to [0,1],\qquad
   Q_n(\tau)=P_n(\tau)\ast \delta_{-S_n(\tau)},
\end{align*}
where $\delta_Y$ denotes the random Dirac measure in $Y\in L^0(\Omega;V)$.
In a first and main step we show that for each $\epsilon>0$ there exists
a compact set
$
  \{\phi_{n,\tau}:\, \tau\in\Upsilon(F_n),\, 1\le \tau\le \sigma_n,\,n\in\N  \}
$
of deterministic $S$-operators $\phi_{n,\tau}\in \SS(V)$, such that for every $n\in\N$ and each $\tau\in \Upsilon(F_n)$
with $ 1\le \tau\le \sigma_n$  we have:
\begin{align}\label{eq.phi_Q}
  1-\Re E\left[\chi_{Q_n(\tau)}(v) \right]\le \scapro{\phi_{n,\tau}v}{v}+4\epsilon
  \qquad\text{for all }v\in V.
\end{align}
For this purpose, fix $\epsilon>0$ and  define the symmetrisation  $\widetilde{P}_{n,k}:=P_{n,k}\ast P_{n,k}^{-}$ where $P_{n,k}^-(B,\omega):=P_{n,k}(-B,\omega)$ for all $B\in \Borel(V)$ and $\omega\in \Omega$.
Define for each $\tau\in \Upsilon(F_n)$ with $\tau\ge 1$ and $n\in\N$ the random measure
\begin{align*}
 \overline{P}_n(\tau)\colon\Borel(V)\times\Omega\to \Rp,\qquad \overline{P}_n(\tau)=\widetilde{P}_{n,1}+\dots + \widetilde{P}_{n,\tau},
\end{align*}
and the random $S$-operator
\begin{align*}
T_n(\tau)\colon V\times\Omega\to V,\qquad  \scapro{T_n(\tau)v}{v}=\int_{B_V}\scapro{v}{h}^2\,\overline{P}_n(\tau)(dh).
\end{align*}
As $\{P_n(\sigma_n):\, n\in\N\}$ is tight it follows that the set $\{\widetilde{P}_{n,1}\ast\dots \ast \widetilde{P}_{n,\sigma_n}:\, n\in\N\}$ is also tight. Part a) of Lemma \ref{le.randomSop} implies that $\{T_n(\tau):\, \tau\in \Upsilon(F_n),\,1\le \tau\le \sigma_n,\, n\in\N\}$ is a tight set of random $S$-operators and that $\{\overline{P}_n(\tau)(\cdot\cap B_V^c):\,\tau\in \Upsilon(F_n),\,1\le \tau\le \sigma_n,\, n\in\N\}$ is a tight set of random measures. It follows from \eqref{eq.rm-tight-bounded} and from \eqref{eq.S-tight}, respectively, that there are constants $c_1$, $c_2>0$ such that for each $n\in\N$ we have
 $P\big(\overline{P}_n(\tau)(B_V^c)>c_1\big)\le \epsilon$ and $P(\tr{T_n(\tau)}>c_2)\le \epsilon$ for all
 $\tau\in \Upsilon(F_n)$ with $1\le \tau\le \sigma_n$. Define for each $n\in\N$ the stopping times
\begin{align*}
  \rho_n^\prime:=\inf\left\{k\in\N:\,  \overline{P}_n(k)(B_V^c)>c_1\right\},\qquad
  \rho_n^{\prime\prime}:=\inf\left\{k\in\N:\, \tr{T_n(k)}>c_2\right\}.
\end{align*}
The stopping time $\rho_n:=\rho_n^\prime\wedge \rho_n^{\prime\prime}$ satisfies for each $n\in \N$ that $P(\rho_n < \tau)\le 2\epsilon$ for all $\tau\in \Upsilon(F_n)$ with $1\le \tau\le \sigma_n$. Lemma \ref{le.submartingale} implies for each  $v\in V$ that
\begin{align*}
1-\Re E[\chi_{Q_n(\tau)}(v)]
&\le E\big[\big(1-\Re \chi_{Q_n(\tau)}(v)\big)\1_{\{\tau\le \rho_n\}}  \big]
+ P(\rho_n<\tau)\\
&\le 1-\Re E\big[\chi_{Q_n(\rho_n\wedge \tau)}(v)\big]+ 2\epsilon\\
&\le E\left[\sum_{k=1}^{\rho_n\wedge\tau} \big(1- \chi_{\widetilde{P}_{n,k}}(v)\big)\right]+2\epsilon.
\end{align*}
The assumed tightness of $\{P_n(\tau):\, \tau\in \Upsilon(F_n),\, 1\le \tau\le \sigma_n,\, n\in\N\}$
yields tightness of $\{\widetilde{P}_{n,1}\ast\dots \ast \widetilde{P}_{n,\rho_n\wedge \sigma_n}:\, n\in\N\}$.
Moreover, since
\begin{align*}
\overline{P}_n(\rho_n\wedge \sigma_n)(B_V^c)
\le \overline{P}_{n}(\rho_n-1)(B_V^c)+ \widetilde{P}_{n,\rho_n}(B_V^c)
\le c_1+1\qquad\text{for all }n\in\N,
\end{align*}
part c) of Lemma \ref{le.randomSop} guarantees tightness of $\{E[\overline{P}_n(\rho_n\wedge \tau)](\cdot\cap B_V^c):\,\tau\in \Upsilon(F_n),\, 1\le \tau\le \sigma_n, \, n\in\N\}$. Thus, there exists a constant $d > 1$ such that
\begin{align}\label{eq.Eoverlinep}
  E\big[\overline{P}_n(\rho_n\wedge \tau)(\{v\in V:\,\norm{v}>d\})\big]\le \epsilon.
\end{align}
Define for each $\tau\in \Upsilon(F_n)$ and $n\in\N$ the random operator
\begin{align*}
R_n(\tau)\colon V\times\Omega\to V,\qquad  \scapro{R_n(\tau)v}{v}=\int_{1<\norm{h}\le d}\scapro{v}{h}^2\,\overline{P}_n(\tau)(dh).
\end{align*}
Since $1-\cos \beta \le 2\beta^2 $ for all $\beta\in \R$ we obtain by \eqref{eq.Eoverlinep} for all $v\in V$ that
\begin{align*}
&E\left[\sum_{k=1}^{\rho_n\wedge \tau} \big(1- \chi_{\widetilde{P}_{n,k}}(v)\big)\right]\\
&\qquad\qquad = E\left[\int_V \big(1-\cos(\scapro{v}{h})\big)\,\overline{P}_n(\rho_n\wedge \tau)(dh) \right]\\
&\qquad\qquad\le 2E\left[\int_{\norm{h}\le d}\scapro{v}{h}^2\,\overline{P}_n(\rho_n\wedge \tau)(dh) \right]
  + 2 E\left[\overline{P}_n(\rho_n\wedge \tau)(\{h\in V:\, \norm{h}>d\})\right]\\
&\qquad\qquad\le 2E\big[ \scapro{T_n(\rho_n\wedge \tau)v}{v}\big] +
2E\big[ \scapro{R_n(\rho_n\wedge \tau)v}{v}\big] + 2\epsilon\\
&\qquad\qquad= 2\big\langle E\big[T_n(\rho_n\wedge \tau)\big]v\big\rangle\big\langle v\big\rangle +
2\big\langle E\big[ R_n(\rho_n\wedge \tau)\big]v\big\rangle\big\langle v\big\rangle+2\epsilon.
\end{align*}
In the last line we applied part (b)  of Lemma \ref{le.randomSop}, which can be done as the definition of the stopping times $\rho_n^\prime$ and $\rho_n^{\prime\prime}$ guarantees for all $n\in\N$ that
\begin{align*}
\tr{R_n(\rho_n^{\prime}\wedge\sigma_n)}
&\le \int_{1<\norm{h}\le d}\norm{h}^2\,\overline{P}_n(\rho_n^{\prime} -1)(dh)
+ \int_{1<\norm{h}\le d}\norm{h}^2\,\widetilde{P}_{n,\rho_n^{\prime}}(dh)\\
&\le d^2 (c_1+1),
\end{align*}
and analogously
\begin{align*}
\tr{T_n(\rho_n^{\prime\prime}\wedge\sigma_n)}
\le \tr{T_n(\rho_n^{\prime\prime}-1)}
   +  \int_{\norm{h}\le 1}\norm{h}^2\,\widetilde{P}_{n,\rho_n^{\prime\prime}}(dh)
\le c_2+1 .
\end{align*}
Moreover, Lemma \ref{le.randomSop} guarantees that the sets $\{E[T_n(\rho_n\wedge \tau)]:\, \tau\in \Upsilon(F_n),\, 1\le \tau\le \sigma_n, \, n\in\N\}$ and  $\{E[R_n(\rho_n\wedge\tau)]:\,\tau\in \Upsilon(F_n),\, 1\le \tau\le \sigma_n, \, n\in\N \}$ are relatively compact in $\L_S(V)$, which completes the proof of \eqref{eq.phi_Q}.

It follows from \eqref{eq.phi_Q} by Theorem VI.2.3 in \cite{Para} that the set $\{E[Q_n(\tau)]:\, \tau\in \Upsilon(F_n),\, 1\le \tau\le \sigma_n, \, n\in\N\}$ and thus also the set $\{Q_n(\tau):\, \tau\in \Upsilon(F_n),\, 1\le \tau\le \sigma_n, \, n\in\N\}$ are tight. Since each random probability measure $Q_n(\tau)$ is the convolution of $P_n(\tau)$ and the random Dirac measure $\delta_{-S_n(\tau)}$, and since the set $\{P_n(\tau):\, \tau\in \Upsilon(F_n),\, 1\le \tau\le \sigma_n, \, n\in\N\}$ is assumed to be tight, it follows that the set
$\{\delta_{-S_n(\tau)}:\,\tau\in \Upsilon(F_n),\, 1\le \tau\le \sigma_n, \, n\in\N\}$  is tight, which completes the proof.
\end{proof}

The following two results are used in the proof of Theorem \ref{th.tightness-tangent}.
\begin{lemma}\label{le.submartingale}
In the setting of Theorem \ref{th.tightness-tangent} define for some $n\in\N$ and for a stopping time $\tau\in \Upsilon(F_n)$ the random probability measure
\begin{align*}
  Q_n(\tau)\colon\Borel(V)\times\Omega\to [0,1],\qquad
   Q_n(\tau)=P_{n,1}\ast\dots \ast P_{n,\tau}\ast \delta_{-S_n(\tau)},
\end{align*}
where $S_n(\tau):=X_{n,1}+\dots +X_{n,\tau}$. Then it follows that
\begin{align*}
  1-\Re{E[\chi_{Q_n(\tau)}(v)]}\le E\left[\sum_{k=1}^{\tau} \big(1- \chi_{\widetilde{P}_{n,k}}(v)\big)\right]
  \qquad\text{for every }v\in V.
\end{align*}
\end{lemma}
\begin{proof} Fix $v\in V$, $n\in\N$,  and define for each $j\in\N$ the $\F_{n,j}$-measurable  random variable
\begin{align*}
  X_n(j):=1-\chi_{\widetilde{P}_{n,j}}(v)-\Re \chi_{Q_n(j-1)}(v)+\Re \chi_{Q_n(j)}(v),
\end{align*}
where we set $Q_n(0)=\delta_{0}$.
We claim that $Y_n(k):=X_n(1)+\dots + X_n(k)$ defines a submartingale
$(Y_n(k):\,k\in\N)$ with respect to $F_n$. For each $k\in\N$ we obtain
\begin{align}\label{eq.condExpsub}
 \! E[X_n(k)\mid\F_{n,k-1}]&=\!
  1-\chi_{\widetilde{P}_{n,k}}(v)-\Re \chi_{Q_n(k-1)}(v)\!+\Re E\left[ \chi_{Q_n(k)}(v)\mid \F_{n,k-1}\right].
\end{align}
Since the random measure $Q_n(j)$ is defined as a convolution its characteristic function obeys
\begin{align*}
  \chi_{Q_n(j)}(v)=
\chi_{\delta_{-S_n(j)}}(v)   \chi_{P_n(j)}(v)
  =e^{-i\scapro{v}{S_n(j)}} \chi_{P_n(j)}(v)
  \qquad\text{for all }j\in\N,
\end{align*}
where $P_n(k):=P_{n,1}\ast\dots \ast P_{n,k}$.
Consequently, we arrive at
\begin{align*}
E\left[ \chi_{Q_n(k)}(v)\mid \F_{n,k-1}\right]  &=  \chi_{P_n(k)}(v) e^{-i\scapro{v}{S_n(k-1)}} E\left[ e^{-i\scapro{v}{X_{n,k}}}\mid \F_{n,k-1}\right]\\
 &= \chi_{Q_n(k-1)}(v) \chi_{P_{n,k}}(v) \chi_{P^{-}_{n,k}}(v)\\
 &=\chi_{Q_n(k-1)}(v) \chi_{\widetilde{P}_{n,k}}(v).
\end{align*}
Applying this equality to \eqref{eq.condExpsub} we obtain
\begin{align*}
  E[X_n(k)\mid\F_{n,k-1}]
  =\big( 1-\Re \chi_{Q_n(k-1)}(v)\big)\big( 1- \chi_{\widetilde{P}_{n,k}}(v)\big).
\end{align*}
Since the last line is non-negative it follows that $(Y_n(k):\,k\in\N)$ is a submartingale.

We conclude from \eqref{eq.fubini-randomE} that
\begin{align*}
E[\chi_{Q_n(1)}(v)]
=E[\chi_{P_{n,1}}(v)] E[e^{-i\scapro{ X_{n,1} }{v}}]
=\chi_{X_{n,1}}(v) \chi_{-X_{n,1}}(v)= E[\chi_{\widetilde{P}_{n,1}}(v)],
\end{align*}
which yields $E[Y_n(1)]=0$.
Doob's optional stopping theorem shows for  $\tau\in \Upsilon(F_n)$  that
$E\big[ Y_n(\tau\wedge N)\big] \ge E\big [Y_n(1)\big]=0$  for all $N\in\N$,
which, by the very definition of $Y_n(k)$,  results in
\begin{align*}
 E\left[ \sum_{k=1}^{\tau\wedge N} \left(1-\chi_{\widetilde{P}_{n,k}}(v) \right)\right] &\ge
 E\left[ -\sum_{k=1}^{\tau\wedge N} \left(\Re \chi_{Q_n(j)}(v)-\Re \chi_{Q_n(j-1)}(v)\right)\right] \\
&= E\bigg[ 1- \Re \chi_{Q_n(\tau\wedge N)}(v)\bigg] .
\end{align*}
Applying the result on monotone convergence to the left hand and Lebesgue's theorem of dominated convergence to the right hand completes the proof.
\end{proof}

\begin{lemma}\label{le.randomSop}
For each $n\in\N$ let $\{M_{n,k}:\,k\in\N\}$ be a sequence of symmetric random probability measures adapted to a filtration $F_n:=\{\F_{n,k}:\,k\in\N_0\}$.
For $\tau\in \Upsilon(F_n)$ with $\tau \ge 1$ denote $\overline{M}_n(\tau):=M_{n,1}+\dots + M_{n,\tau}$ and define for some $c>0$  the random operators
\begin{align*}
T_n(\tau)\colon V\times\Omega\to V,\qquad  \scapro{T_n(\tau)v}{v}=\int_{\norm{h}\le c}\scapro{v}{h}^2\,\overline{M}_n(\tau)(dh),
\end{align*}
and the random measures
\begin{align*}
  N_n(\tau)\colon\Borel(V)\times\Omega\to \Rp,\qquad
   N_n(\tau)(B)=\overline{M}_n(\tau)\big(B\cap \{ \norm{v}>c\}\big).
\end{align*}
If there exists a sequence $\{\sigma_n:\, n\in\N\}$ of finite stopping times $\sigma_n\in \Upsilon(F_n)$ with $\sigma_n\ge 1$
such that $\{M_{n,1}\ast\dots\ast M_{n,\sigma_n}:\,n\in\N\}$ is tight, then
we have:
\begin{enumerate}
\item[{\rm (a)}]
the set $\{T_n(\tau):\, \tau\in \Upsilon(F_n),\,1\le\tau\le \sigma_n,\, n\in\N\}$ of random $S$-operators and the
 set $\{N_n(\tau):\,\tau\in \Upsilon(F_n),\,1\le\tau\le \sigma_n,\, n\in\N \}$ of random measures are tight.
\item[{\rm (b)}] uniform integrability of
$\{\tr{T_n(\sigma_n)}:\, n\in\N\}$ implies that, for $\tau\in \Upsilon(F_n)$,
\begin{align*}
E[T_n(\tau)]\colon V \to V,\qquad  \scapro{E[T_n(\tau)]v}{v}=E[\scapro{T_n(\tau)v}{v}]
\end{align*}
defines a relatively compact set $\{E[T_n(\tau)]:\,\tau\in \Upsilon(F_n),\,1\le\tau\le \sigma_n,\, n\in\N\}$ of $S$-operators.
\item[{\rm (c)}]
 uniform integrability of
$\{N_n(\sigma_n)(V):\, n\in\N\}$ implies that
$\{E[N_n(\tau)]:\, \tau\in \Upsilon(F_n),\,1\le\tau\le \sigma_n,\, n\in\N\}$ is relatively compact.

\end{enumerate}
\end{lemma}
\begin{proof}
(a) Let $R_n$ denote the  infinitely divisible random measure with characteristic function
\begin{align*}
\chi_{R_n}\colon V\times\Omega\to \C, \qquad
  \chi_{R_n}(v)=\exp\left( \int_V \left(e^{i\scapro{h}{v}}-1\right)\, \overline{M}_n(\sigma_n)(dh)\right).
\end{align*}
The inequality
\begin{align*}
  1-\exp\left( \sum_{k=1}^n(\beta_k-1)\right)\le 1-\prod_{k=1}^n \beta_k,
  \qquad\text{for all }\beta_k\in [0,1],\, n\in\N,
\end{align*}
yields for every $v\in V$ and $n\in\N$  the estimate
\begin{align*}
  1-E\left[\chi_{R_n}(v)\right]
  &=  E\left[1- \exp\left( \sum_{k=1}^{\sigma_n}\left(\chi_{M_{n,k}}(v)-1\right)\right)\right]
\le 1-E\left[\prod_{k=1}^{\sigma_n} \chi_{M_{n,k}}(v)\right].
\end{align*}
Tightness of $\{M_{n,1}\ast\dots\ast M_{n,\sigma_n}:\,n\in\N\}$ implies by Theorem VI.2.3
in \cite{Para} together with \eqref{eq.fubini-randomE} that the set $\{E[R_n]:\, n\in\N\}$ is tight. It follows that for each $\epsilon> 0$ there exists for every $n\in\N$ a set $A_n\in\A$  with $P(A_n)>1-\epsilon$  such that the set
\begin{align*}
  \big\{\1_{A_n}(\omega)R_n(\cdot,\omega)+\1_{A_n^c}(\omega)\delta_0(\cdot):\, \omega\in \Omega,\,n\in\N\big\}
\end{align*}
of infinitely divisible probability measures is relatively compact.
Theorem VI.5.1 in \cite{Para} implies that the set $\big\{T_n(\sigma_n)(\omega):\, \omega\in A_n,\, n\in\N\big\}$ is compact and the set $\big\{N_n(\sigma_n)(\cdot,\omega):\, \omega\in A_n,\, n\in\N\big\}$ is relatively compact. The monotonicity $\scapro{T_n(\tau)f_k}{f_k}\le\scapro{T_n(\sigma_n)f_k}{f_k}$ for all $k\in\N$ and $N_n(\tau)\le N_n(\sigma_n)$ for each $\tau\le \sigma_n$ completes  the proof by \eqref{eq.rm-tight-bounded}, \eqref{eq.rm-tight-tight}
and \eqref{eq.compactS-finite-trace}, \eqref{eq.compactS-sup-limit}.

(b) By applying Tonelli's theorem we obtain
\begin{align}\label{eq.compact-A-1}
  \sup_{n\in\N}\tr{ E[T_n(\sigma_n)]}
  = \sup_{n\in\N}E\left[\sum_{k=1}^\infty \scapro{T_n(\sigma_n)f_k}{f_k}\right]
  = \sup_{n\in\N} E\big[\tr{T_n(\sigma_n)}\big]<\infty.
\end{align}
Let $\epsilon>0$ be given and choose $\delta>0$ such that $P(A)\le \delta$ for any $A\in\A$  implies $E[\tr{T_n(\sigma_n)}\1_A]\le \epsilon$ for all $n\in\N$. From part (a) it follows by \eqref{eq.compactS-sup-limit} that there are $A_n\in \A$,  $n\in\N$, with $P(A_n)>1-\delta$ and $N_0\in \N$ such that
\begin{align*}
\sup_{n\in\N} \sup_{\omega\in A_n}\sum_{k=N_0}^\infty \scapro{T_n(\sigma_n)(\omega)f_k}{f_k}\le\epsilon.
\end{align*}
It follows that
\begin{align*}
\sup_{n\in\N}\sum_{k=N_0}^\infty \scapro{E[T_n(\sigma_n)]f_k}{f_k}
&\le \sup_{n\in\N} E\left[\1_{A_n} \sum_{k=N_0}^\infty \scapro{T_n(\sigma_n)f_k}{f_k}\right] +E\left[\tr{T_n(\sigma_n)}\1_{A_n}^c\right]\\
&\le 2\epsilon,
\end{align*}
which shows that
\begin{align}\label{eq.compact-A-2}
  \lim_{N\to\infty} \sup_{n\in\N} \sum_{k=N}^\infty \scapro{E[T_n(\sigma_n)]f_k}{f_k}=0.
\end{align}
Both properties \eqref{eq.compact-A-1} and \eqref{eq.compact-A-2} establish that
$\{E[T_n(\sigma_n)]:\, n\in\N\}$ is relatively compact, which completes the proof by monotonicity
$\scapro{E[T_n(\tau)]f_k}{f_k}\le\scapro{E[T_n(\sigma_n)]f_k}{f_k}$ for all $k\in\N$ for $\tau \le \sigma_n$ by \eqref{eq.compactS-finite-trace} and \eqref{eq.compactS-sup-limit}.

(c) Can be proved as (b).
\end{proof}

\section{Radonification of the increments}\label{se.increments}

In this section we solve the problem (1) and (2) mentioned in the Introduction.
Recall the definition $p(X)=E[1\wedge \norm{X}^2]$ for any $V$-valued random variable $X$ in \eqref{eq.def-metric-L0}. The following inequality originates from the work \cite{Jakubowskietal},
but since we only need a special case we give a short proof here.
\begin{lemma}\label{le.Gauss-dom}
  There exists a universal constant $c>0$ such that for any $\phi\in \L_2(U,V)$ and any cylindrical random variable $Z\colon U\to L^0_P(\Omega;\R)$ we have
   \begin{align*}
    p(\phi Z)\le c \int_U p(Zu)\, (\gamma\circ (\phi^{\ast})^{-1})(du),
   \end{align*}
where $\gamma$ denotes the canonical Gaussian cylindrical measure on $V$.
\end{lemma}
\begin{proof}
Let $(\Omega^\prime, \A^\prime, P^\prime)$ be another probability space and
let $\Gamma\colon V\to L_{P^\prime}^0(\Omega^\prime;\R)$ be a cylindrical random variable distributed according to the canonical cylindrical Gaussian distribution $\gamma$ on $V$. From the inequality
\begin{align*}
(1\wedge \abs{\alpha})(1\wedge \abs{\beta})\le 1 \wedge \abs{\alpha\beta}
  \qquad\text{for all }\alpha,\beta \in\R,
\end{align*}
it follows for a real-valued, standard normally distributed random variable $\xi$
that
\begin{align*}
\left(1\wedge \norm{v}^2\right) E^\prime\!\left[1\wedge \abs{\xi}^2\right]\le
   E^\prime\!\left[1\wedge\norm{v}^2\abs{\xi}^2\right]
   = E^\prime\!\Big[1\wedge \abs{\Gamma v}^2\Big]
   \qquad\text{for all }v\in V.
\end{align*}
Consequently, by defining $c:=\left(E^\prime\!\left[1\wedge\abs{\xi}^2\right]\right)^{-1}$ we obtain
\begin{align*}
 (1\wedge \norm{v}^2) \le c  E^\prime\!\Big[1\wedge \abs{\Gamma v}^2\Big]\qquad
  \text{for all }v\in V.
\end{align*}
 It follows that
\begin{align*}
p(\phi Z)=E\left[ 1\wedge \norm{\phi Z}^2\right]
  &\le c E\left[ E^\prime\!\left[[ 1\wedge\abs{\Gamma(\phi Z)}^2\right]\right]\\
 &= c E\left[\int_V \left( 1\wedge  \scapro{v}{\phi Z}^2\right)\,\gamma(dv)\right]\\
 &= c \int_V E\left[ 1\wedge  \abs{Z(\phi^\ast v)}^2 \right]\,\gamma(dv)\\
  &= c\int_U E\left[ 1\wedge  \abs{Zu}^2\right]\, (\gamma\circ (\phi^{\ast})^{-1})(du),
\end{align*}
which completes the proof.
\end{proof}

Let $(L(t):\, t\ge 0)$ be a cylindrical L{\'e}vy process in $U$. We equip the probability space with the filtration generated by $L$ and defined by
\begin{align*}
  \F_t:=\sigma(\{L(s)u:\, u\in U, \, s\in [0,t]\})\qquad\text{for all }t\ge 0.
\end{align*}
Fix the times $0\le s\le t$.
An $\L_2(U,V)$-valued, $\F_s$-measurable random variable $\Phi$
is called {\em simple} if it is of the form
\begin{align}\label{eq.simple}
  \Phi(\omega)=\sum_{i=1}^m \1_{A_i}(\omega) \phi_i\qquad\text{for all }\omega\in\Omega,
\end{align}
for disjoint sets $A_1,\dots, A_m\in \F_s$ and $\phi_1,\dots, \phi_m\in \L_2(U,V)$. The space of all $\L_2(U,V)$-valued, $\F_s$-measurable, simple random variables is
denoted by $S(\Omega, \F_s;\L_2)$.
It follows from \eqref{eq.HSradonifying} that for each $i=1,\dots, m$ there
exists an $V$-valued random variable $\phi_i\big(L(t)-L(s)\big)$ satisfying
\begin{align*}
  \scapro{\phi_i\big(L(t)-L(s)\big)}{v}
  =\big(L(t)-L(s)\big)(\phi_i^\ast v)
  \qquad\text{for all }v\in V.
\end{align*}
Define an $\F_s$-measurable, $V$-valued random variable by
\begin{align*}
  J(\Phi):=\sum_{i=1}^m \1_{A_i} \phi_i\big(L(t)-L(s)\big).
\end{align*}
In this situation  we define
\begin{align*}
\big(L(t)-L(s)\big)(\Phi^\ast v):=
\scapro{J(\Phi)}{v}
\qquad\text{for all }v\in V.
 \end{align*}
 The following result enables us to extend this definition of {\em radonified increments} from simple to arbitrary random variables $\Phi$.
\begin{theorem}\label{th.JPHI}
Let $0\le s\le t$ be fixed. For each $\F_s$-measurable, $\L_2(U,V)$-valued random variable $\Phi$ there exist an $V$-valued random variable $Y$ and
   a sequence $\{\Phi_n\}_{n\in \N}$ of simple random variables in
  $S(\Omega,\F_s;\L_2)$ with $\Phi_n \to \Phi$ $P$-a.s. such that
  \begin{align*}
    Y=\lim_{n\to \infty} J(\Phi_n)
    \quad\text{ in probability}.
  \end{align*}
Moreover, the limit $Y$ does not depend on the sequence $\{\Phi_n\}_{n\in \N}$.
\end{theorem}
\begin{proof}
  Since $\Phi\colon \Omega\to \L_2(U,V)$ is strongly $\F_s$-measurable there exists a sequence
  $\{\Phi_n\}_{n\in \N}$ of simple random variables in
  $S(\Omega,\F_s;\L_2)$ with $\Phi_n\to\Phi$ $P$-a.s. It remains to show that $(J(\Phi_n))_{n\in\N}$ is a Cauchy sequence in $L_P^0(\Omega;V)$. By linearity it is sufficient to show that $\Phi_n\to 0$ $P$-a.s. for $n\to\infty$ implies that $J(\Phi_n)\to 0$ in probability for $n\to\infty$. For this purpose assume that
\begin{align*}
  \Phi_n(\omega)=\sum_{i=1}^{m_n} \1_{A_{n,i}}(\omega) \phi_{n,i}\qquad\text{for all }\omega\in\Omega,
\end{align*}
for disjoint sets $A_{n,1},\dots, A_{n,m_n}\in \F_s$ and $\phi_{n,1},\dots, \phi_{n,m_n}\in \L_2(U,V)$, $m_n\in\N$ and $\Phi_n\to 0$ $P$-a.s. for $n\to\infty$.
Define the cylindrical random variable $Z:=L(t)-L(s)$. Independence of $Z$ and $\F_s$
implies that $\phi_{n,i}Z$ is also independent of $\F_s$. Using this independence, Lemma \ref{le.Gauss-dom} and \eqref{eq.fubini-randomE}, we obtain for each $n\in\N$ that
\begin{align*}
 p(J(\Phi_n))
 &= E\left[1\wedge \norm{\sum_{i=1}^{m_n} \1_{A_{n,i}} \phi_{n,i}Z }^2  \right]\\
&= E\left[\sum_{i=1}^{m_n} \1_{A_{n,i}} \left(1\wedge \norm{\phi_{n,i}Z }^2 \right) \right]\\
&= \int_{\Omega} \sum_{i=1}^{m_n}\1_{A_{n,i}}(\omega)
    p(\phi_{n,i}Z )\, P(d\omega)\\
&\le c \int_{\Omega} \sum_{i=1}^{m_n}\1_{A_{n,i}}(\omega)
  \left(\int_U p(Zu)\, (\gamma\circ (\phi_{n,i}^{\ast})^{-1})(du)\right)  \, P(d\omega)\\
&= c \int_{\Omega}
  \left(\int_U p(Zu)\, M_n(du,\omega)\right)  \, P(d\omega)\\
&= c \int_U p(Zu)\, E[M_n](du),
\end{align*}
where $c$ denotes the constant derived in Lemma \ref{le.Gauss-dom} and
$M_n$ is the random probability measure defined by
\begin{align*}
M_n\colon\Borel(U)\times \Omega\to [0,1],
\qquad  M_n(B,\omega)= \sum_{i=1}^{m_n}\1_{A_{n,i}}(\omega) (\gamma\circ (\phi_{n,i}^{\ast})^{-1})(B).
\end{align*}
Recall the definition $E[M_n](B):=E[M_n(B)]$ for all $B\in \Borel(U)$.
For each $n\in \N$ and $\omega\in \Omega$ the measure
 $M_n(\cdot,\omega)$ is  Gaussian with expectation $0$ and covariance operator
\begin{align*}
Q_n(\omega)\colon U\to U,\qquad
Q_n(\omega)= \sum_{i=1}^{m_n}\1_{A_{i,n}}(\omega)\phi_{n,i}^\ast\phi_{n,i}.
\end{align*}
Consequently, we have for all $u\in U$  that
\begin{align}\label{eq.char-Gauss-random}
  E[\phi_{M_n}(u)]=E[e^{i\scapro{Q_nu}{u}}]
  =E[e^{i\norm{\Phi_n u}_V^2}]\to 1\quad
  \text{as }n\to\infty.
\end{align}
Egorov's theorem implies that for each $\epsilon>0$ there exists a set $A\in\A$
with $P(A)>1-\epsilon$ such that
\begin{align*}
  \sup_{n\in\N}\sup_{\omega\in \A}\tr{Q_n(\omega)}<\infty.
\end{align*}
As $Q_n(\omega)$ is the covariance operator of the Gaussian measure $M_n(\cdot,\omega)$ it follows that the set $\{M_n(\cdot,\omega):\,\omega\in A, \,n\in\N\}$ is tight. Thus, the set $\{M_n:\,n\in\N\}$ of random measures is tight, which implies together with \eqref{eq.char-Gauss-random} that
$E[M_n]$ converges weakly to the Dirac measure in $0$.
Since the function $u\mapsto p(Zu)$ is bounded and continuous we obtain
\begin{align*}
\int_H p(Zu)\, E[M_n](du)
\to 0 \quad \text{as }n\to\infty.
\end{align*}
Linearity of $J$ guarantees the claimed uniqueness,
which completes the proof.
\end{proof}

Theorem \ref{th.JPHI} enables us to define for each $0\le s\le t$ and  $\F_s$-measurable random variable $\Phi\colon\Omega\to \L_2(U,V)$ the $V$-valued random variable
\begin{align*}
\Phi\big(L(t)-L(s)\big):=\lim_{n\to\infty}J(\Phi_n),
\end{align*}
where $(\Phi_n)_{n\in\N}\subseteq S(\Omega,\F_s;\L_2)$ converges to $\Phi$ $P$-a.s.
We define then the increments of the cylindrical L{\'e}vy process $L$ under the random mapping
$\Phi$ by
\begin{align*}
  \big(L(t)-L(s)\big)(\Phi^\ast v):=
  \left\langle \Phi\big(L(t)-L(s)\big),\,v \right\rangle
  \qquad\text{for all }v\in V.
\end{align*}

We finish this section with calculating the conditional characteristic function of the
radonified increments.
\begin{lemma}\label{le.conditional-exp}
If $0\le s\le t$ and $\Phi\colon\Omega\to \L_2(U,V)$ is an $\F_s$-measurable random variable then it follows for each $v\in V$ that
\begin{align}\label{eq.conditional-inc}
&  E\left[ \exp\left(i \scapro{\Phi\big(L(t)-L(s)\big)}{v}\right)\Big| \F_s\right]=\exp\big((t-s)S(\Phi^\ast v)\big)\qquad\text{$P$-a.s.,}
\end{align}
where $S\colon U\to\C$ denotes the cylindrical L{\'e}vy symbol of $L$
defined in \eqref{eq.Levy-symbol}.
\end{lemma}
\begin{proof}

If $\Phi$ is simple then it is easy to establish the  equality claimed in \eqref{eq.conditional-inc}. For an arbitrary $\F_s$-measurable random variable
$\Phi\colon\Omega\to \L_2(U,V)$, Theorem \ref{th.JPHI} guarantees that  there exists a sequence $(\Phi_n)_{n\in\N}$ of simple random variables in $\S(\Omega,\F_s;\L_2)$
satisfying $\Phi_n\to \Phi$ $P$-a.s. as $n\to\infty$ and for all $v\in V$:
\begin{align}\label{eq.limit-left}
\lim_{n\to\infty} \big( L(t)-L(s)\big)(\Phi^\ast_n v)
=\big( L(t)-L(s)\big)(\Phi^\ast v)
\quad\text{ in probability}.
\end{align}
On the other hand, as $\Phi_n\to \Phi$ $P$-a.s. the continuity of  the cylindrical L{\'e}vy symbol $S:U\to \C$ yields for all $v\in V$
\begin{align}\label{eq.limit-right}
\lim_{n\to\infty} \exp\big((t-s)S\left(\Phi^\ast_n v \right)\big)
= \exp\big((t-s)S\left( \Phi^\ast v\right)\big)
\quad \text{$P$-a.s.}
\end{align}
The equations \eqref{eq.limit-left} and \eqref{eq.limit-right} show that the relation
\eqref{eq.conditional-inc} can be generalised to arbitrary  $\Phi\in L_P^0(\Omega,\F_s;\L_2)$.
\end{proof}

\section{The stochastic integral}\label{se.stochastic-integral}

We begin the definition of the stochastic integral very classical with simple integrands.
An $\L_2(U,V)$-valued, stochastic process $(\Psi(t):\, t\in [0,T])$ is called {\em simple} if it is of the form
\begin{align}\label{eq.defsimple}
  \Psi(t)=\Phi_0\1_{\{0\}}(0)+ \sum_{j=1}^{N} \Phi_j \1_{(t_j,t_{j+1}]}(t)
  \qquad\text{for all }t\in [0,T],
\end{align}
where $0=t_1< \cdots < t_{N+1}=T$ is a finite sequence of
deterministic times  and each $\Phi_j\colon \Omega\to \L_2(U,V)$ is an
$\F_{t_j}$-measurable random variable for each $j=0,\dots, N$.
The set of all simple $\L_2(U,V)$-valued stochastic processes is denoted by $\H_0(U,V)$.

Let $(\Psi(t):\,t\in [0,T])$ be a simple process in $\H_0(U,V)$ of the form \eqref{eq.defsimple} and $(L(t):\,t\ge 0)$ be
a cylindrical L{\'e}vy process in $U$.  Theorem \ref{th.JPHI} guarantees that for each $j=1,\dots, N$ and $t\in [0,T]$, there exists the random variable
\begin{align*}
  J(\Phi_j)(t):=\Phi_j\big(L(t\wedge t_{j+1})-L(t\wedge t_j)\big)\colon \Omega\to V,
\end{align*}
satisfying
\begin{align*}
  \big(L(t\wedge t_{j+1})-L(t\wedge t_j)\big)(\Phi_j^\ast v)
  =\scaprob{J(\Phi_j)(t)}{v}\qquad\text{for all }v\in V.
\end{align*}
Thus,  we can  define a random variable in $L^0_P(\Omega;V)$ for each $t\in [0,T]$ by
\begin{align*}
I(\Psi)(t)\colon\Omega\to V,
 \qquad I(\Psi)(t):=  J(\Phi_1)(t) + \cdots + J(\Phi_{N})(t).
\end{align*}
Obviously, the random variable $I(\Psi)(t)$ obeys
\begin{align*}
 \scaprob{I(\Psi)(t)}{v} = \sum_{j=1}^{N} \scaprob{J(\Phi_j)(t)}{v} = \sum_{j=1}^{N} \big(L(t\wedge t_{j+1})-L(t\wedge t_j)\big)(\Phi_j^\ast v)
\end{align*}
for all $v\in V$ and $t\in [0,T]$.

In the following we extend the domain of $I$ to the linear space
\begin{align*}
\H(U,V):=\big\{\Psi\colon [0,T]\times\Omega\to \L_2(U,V): \text{ adapted
 and with  c{\`a}gl{\`a}d paths}  \big\}.
\end{align*}
The trajectories of an element $\Psi$ in $\H(U,V)$ are in the space
$D_-\big([0,T];\L_2(U,V)\big)$ of $\L_2(U,V)$-valued functions which are
continuous from the left and have limits from the right (c{\`a}gl{\`a}d).
Recall from Section \ref{se.preliminaries} that this space is
equipped with the Skorokhod metric $d_J$ defined in \eqref{de.Skorokhod-metric}. As $\L_2(U,V)$ is separable every $\Psi\in \H(U,V)$ can also be considered as a random variable $\Psi\colon \Omega\to D_-\big([0,T];\L_2(U,V)\big)$.

The definition of the stochastic integral for arbitrary integrands in the space $\H(U,V)$ is given by the following result. The assumed approximation by simple processes is presented in the subsequent Lemma \ref{le.approximation}.
\begin{theorem}\label{th.stoch-op-continuous}
For every sequence $(\Psi_n)_{n\in\N}\subseteq \H_0(U,V)$
which converges to some $\Psi\in \H(U,V)$ in probability in the
Skorokhod metric, i.e.\
\begin{align*}
\lim_{n\to\infty} P\left(d_J\big(\Psi_n,\Psi)\big)\ge \epsilon\right)=0
\qquad\text{for all $\epsilon>0$, }
\end{align*}
there exists a $V$-valued, adapted semi-martingale $(I(\Psi)(t):\,t\in [0,T])$ with  c{\`a}dl{\`a}g trajectories obeying for each $t\in [0,T]$ and $\epsilon>0$:
  \begin{align*}
\lim_{n\to\infty} P\big(\norm{I(\Psi_n)(t)-I(\Psi)(t)}\ge \epsilon\big)=0.
  \end{align*}
 The limit $I(\Psi)$ does not depend on the sequence $(\Psi_n)_{n\in\N}$, i.e.\ it is unique up to evanescence.
\end{theorem}

In this work, we define simple integrands  as stochastic processes which equal a random variable $\Phi_j$ on deterministic  but not random intervals $(t_j,t_{j+1}]$. This guarantees that the radonification of the
increments $J(\Phi_j)(t)$ is well defined by the approach in Section~\ref{se.increments}. If the interval $(t_j,t_{j+1}]$ were random the integrand $\Phi_j$ would not be independent of $L(t_{j+1})-L(t_j)$ and this method  could not be applied any more.

Simple integrands defined on random intervals are dense in $\H(U,V)$ with respect to the uniform convergence on $[0,T]$ in probability, i.e.\ the so-called ucp convergence. In our case of deterministic partitions, we have to weaken the topology to the Skorokhod topology.

\begin{lemma}\label{le.approximation}
  For every $\Psi\in \H(U,V)$ there exists a sequence $(\Psi_n)_{n\in\N}$ of simple stochastic processes $\Psi_n\in \H_0(U,V)$, each defined on a partition $(t_{n,k})_{k=1,\dots, n}$ of the interval $[0,T]$  with $ \max\limits_{j=1,\dots, n}\abs{t_{n,j+1}-t_{n,j}}\to 0$ for $n\to\infty$ and with $\{\Psi_n(t):\, t\in [0,T]\}$ in the closure of $\{\Psi(t):\, t\in [0,T]\}$
  such that
  \begin{align}\label{eq.approx-seq-J}
     \lim_{n\to\infty}d_J\big(\Psi_n(\omega),\Psi(\omega)\big)= 0
     \quad\text{for all }\omega\in \Omega.
  \end{align}
\end{lemma}
\begin{proof}
This follows from the construction of the approximating sequence as the
discretisation in the analogue result for deterministic, c\`adl\`ag functions, see for example \cite[Le.VII.6.5]{Para}.
\end{proof}

\begin{proposition}\label{pro.condcompunderK}
Let $\mu$ be a continuous cylindrical probability  measure on $\Z(U)$ and $K$ be a compact set in $\L_2(U,V)$.  Then the set $\{\mu\circ \phi^{-1}:\, \phi\in K\}$ is relatively compact in the space $\M_1(V)$ of probability measures on $\Borel(V)$.
\end{proposition}
\begin{proof}
According to \cite[Pro.IV.4.2, p.236]{Vaketal} there exist a probability space $(\Omega,\A,P)$ and a cylindrical random variable
$Z\colon U\to L^0_P(\Omega;\R)$ such that $\mu$ is the cylindrical distribution of $Z$.
Then the random variable $\phi(Z)$ is distributed according to the probability measure
$\mu\circ \phi^{-1}$ for each $\phi\in K$.
By \cite[Co.1 in I.3.9, p.52]{Vaketal} the set $\{\mu\circ \phi^{-1}:\, \phi\in K\}$ is relatively compact
if and only if
\begin{align}
 &\text{\rm (i) } \lim_{r\to\infty}\sup_{\phi\in K} P\big( \norm{\phi(Z)}\ge r\big)=0,\label{eq.Condition-cond-compact1}\\
  &\text{\rm (ii) } \lim_{N\to\infty} \sup_{\phi\in K} P\left(\sum_{k=N+1}^\infty \scapro{\phi(Z)}{f_k}^2\ge r\right)=0
  \qquad\text{for each }r>0. \label{eq.Condition-cond-compact2}
 \end{align}
Recall that $(f_k)_{k\in\N}$ denotes the orthonormal basis in $V$.
For fixed $m$ and $N$ in $\N$ with $N<m$ and for $\phi\in K$ define the $m-N$-dimensional random vector $Y:=(\scapro{\phi(Z)}{f_{N+1}},\dots, \scapro{\phi(Z)}{f_m})$. The characteristic  function $\chi_Y\colon\R^{m-N}\to\C$ of $Y$ is given for $\beta=(\beta_{N+1},\dots, \beta_m)\in\R^{m-N}$ by
\begin{align*}
\chi_Y(\beta)&=E\left[\exp\left(i\sum_{k=N+1}^m \beta_k \scapro{\phi(Z)}{f_k}\right)\right]
= \chi_Z \left(\sum_{k=N+1}^m  \phi^\ast(\beta_k f_k) \right).
\end{align*}
Let $\epsilon>0$ be given. The continuity of the characteristic function $\chi_Z\colon U\to\C$ implies that
there exists a $\delta>0$ such that
\begin{align}\label{eq.estimate-char-fct}
\abs{1-\chi_Z (u)}\le \epsilon\1_{B_\delta}(u) + 2\1_{B_\delta^c}(u)
\le \epsilon + 2 \frac{\norm{u}^2}{\delta^2}\qquad\text{for all }u\in U,
\end{align}
where $B_\delta:=\{u\in U:\, \norm{u}\le \delta\}$. By applying \cite[Pro.IV.5.2, p.205]{Vaketal} we obtain for every $r>0$ the inequality
\begin{align}
P\left(\sum_{k=N+1}^m \scapro{\phi(Z)}{f_k}^2\ge r^2\right)
&= P_Y\left(\beta\in\R^{m-N}:\, \abs{\beta}\ge r\right) \notag\\
&\le 3 \int_{\R^{m-N}} \left(1-\chi_Y\left(\tfrac{\beta}{r}\right)\right)\, d\gamma_{m-N}(\beta), \label{eq.inequality-P-char}
\end{align}
where $\gamma_{m-N}$ denotes the standard normal distribution on $\Borel(\R^{m-N})$.
Inequality \eqref{eq.estimate-char-fct} implies
\begin{align}
&\int_{\R^{m-N}} \left(1-\chi_Y\left(\tfrac{\beta}{r}\right)\right)\, d\gamma_{m-N}(\beta)\notag\\
&\qquad = \int_{\R^{m-N}} \left(1-\chi_Z\left(\tfrac{1}{r}\sum_{k=N+1}^m \phi^\ast (\beta_k f_k)\right)\right)\, d\gamma_{m-N}(\beta_{N+1},\dots,\beta_m)\notag\\
&\qquad \le \int_{\R^{m-N}} \left(\epsilon+ \frac{2}{\delta^2 r^2} \norm{\sum_{k=N+1}^m \phi^\ast(\beta_k f_k)}^2\right)\, d\gamma_{m-N}(\beta_{N+1},\dots,\beta_m) \notag\\
&\qquad = \epsilon + \frac{2}{\delta^2 r^2} \sum_{k=N+1}^m\sum_{\ell=N+1}^m
  \scapro{\phi^\ast(f_k)}{\phi^\ast(f_\ell)}\int_{\R^{m-N}} \beta_k\beta_\ell \,d\gamma_{m-N}(\beta_{N+1},\dots,\beta_m) \notag\\
&\qquad = \epsilon + \frac{2}{\delta^2 r^2} \sum_{k=N+1}^m \norm{\phi^\ast f_k}^2. \label{eq.inequality-int-gauss}
\end{align}
By applying the estimate \eqref{eq.inequality-int-gauss} to inequality \eqref{eq.inequality-P-char}  we obtain for every $r>0$:
\begin{align}\label{eq.inequality-P-tailsum}
\sup_{\phi\in K} P\left(\sum_{k=N+1}^\infty \scapro{\phi(Z)}{f_k}^2\ge r\right)
&= \sup_{\phi\in K} \lim_{m\to \infty} P\left(\sum_{k=N+1}^m \scapro{\phi(Z)}{f_k}^2\ge r\right)\notag\\
&\le 3\epsilon + \frac{6}{\delta^2 r^2}\sup_{\phi\in K} \sum_{k=N+1}^\infty \norm{\phi^\ast f_k}^2.
\end{align}
Since the mapping $\phi\mapsto \phi^\ast$ is continuous on $\L_2(U,V)$ the set $\{\phi^\ast:\, \phi\in K\}$
is compact in $\L_2(V,U)$. Thus, Condition \eqref{eq.Condition-cond-compact2} follows from \eqref{eq.inequality-P-tailsum} by the characterisation \eqref{eq.HS-compact2} of
the compact set $K$.
Similarly, the inequalities \eqref{eq.inequality-P-char} and \eqref{eq.inequality-int-gauss}
imply for every $r>0$
\begin{align*}
P\big( \norm{\phi(Z)}\ge r\big)
= \lim_{m\to\infty} P\left(\sum_{k=1}^m \scapro{\phi(Z)}{f_k}^2\ge r^2\right)
\le  3\epsilon + \frac{6}{\delta^2 r^2}\sup_{\phi\in K} \norm{\phi}_{\L_2}^2.
\end{align*}
Thus,  Condition \eqref{eq.Condition-cond-compact1} follows from boundedness of $K$, which completes the proof.
\end{proof}

\begin{lemma}\label{le.infdiv-convex}
If the set $\{\mu_\alpha:\, \alpha\in J\}$ of infinitely divisible probability measures $\mu_\alpha$
on $\Borel(V)$ for an arbitrary index set $J$ is relatively compact in $\M_1(V)$ then
\begin{align*}
& \Big\{\mu_{\alpha_1}^{\ast t_1}\ast \dots \ast \mu_{\alpha_n}^{\ast t_n} :\; \alpha_i\in J,\  t_i\ge 0, \, t_1+\dots +t_n\le T \text{ for }i=1,\dots, n,\, n\in\N  \Big\}
\end{align*}
is also relatively compact in $\M_1(V)$.
\end{lemma}
\begin{proof}
According to \cite[Th.VI.5.3, p.187]{Para} the  set $\{\mu_\alpha:\,\alpha \in J\}$
 of infinitely divisible probability measures $\mu_\alpha$ with characteristics
 $(a_\alpha,Q_\alpha,\nu_\alpha)$ is relatively compact if and only if the following
 three conditions are satisfied:
 \begin{enumerate}
   \item[(1)] the set $\{a_\alpha:\,\alpha\in J\}\subseteq V$ is relatively compact;
   \item[(2)] the set $\{\nu_\alpha:\,\alpha\in J\}$ restricted to the complement of any neighborhood of the origin is relatively compact;
   \item[(3)] the operators $T_\alpha:V\to V$ defined by
   \begin{align*}
       \scapro{T_\alpha v}{v}:=\scapro{Q_\alpha v}{v}+\int_{\norm{h}\le 1} \scapro{v}{h}^2\, \nu_\alpha (dh)
   \end{align*}
    satisfy the conditions
    \begin{align*}
       {\rm (i)}&\;\sup_{\alpha\in J} \sum_{k=1}^\infty \scapro{T_\alpha f_k}{f_k}<\infty,\\
       {\rm (ii)}&\;\lim_{N\to\infty} \sup_{\alpha\in J} \sum_{k=N}^\infty \scapro{T_\alpha f_k}{f_k}=0.
    \end{align*}
 \end{enumerate}
For $\alpha_i\in J$, $t_i\ge 0$ and $t_1+\dots +t_n\le T$,  the infinitely divisible probability  measure $\mu_{\alpha_1}^{\ast t_1}\ast \dots \ast \mu_{\alpha_n}^{\ast t_n}$ has the characteristics
\begin{align*}
  \Big(\sum_{i=1}^n t_i a_{\alpha_i}, \sum_{i=1}^n t_i Q_{\alpha_i},
   \sum_{i=1}^n t_i \nu_{\alpha_i}\Big).
\end{align*}
It remains to show that these characteristics satisfy the corresponding Conditions (1) -- (3) above.
\begin{enumerate}
\item[(1)] Define the set
      \begin{align*}
 D:=\left\{\sum_{i=1}^n t_i a_{\alpha_i}:\, \alpha_i\in J,\  t_i\ge 0, \, t_1+\dots +t_n\le T \text{ for }i=1,\dots, n,\, n\in\N  \right\}.
     \end{align*}
Since the set $A:=\{a_\alpha:\, \alpha \in J\}$ is relatively compact, it follows by a Theorem of S. Mazur that the convex hull co$(A)$ of $A$ is relatively compact.
Since the mapping
\begin{align*}
  m:[0,T]\times \overline{\text{\rm co}(A)}\to V, \qquad m(t,v)=tv,
\end{align*}
 is continuous and $D\subseteq m\big([0,T]\times \overline{\text{\rm co}(A)}\big)$, we can conclude that the set $D$ is relatively compact.
\item[(2)] Prohorov's Theorem guarantees that the set $\{\nu_\alpha:\,\alpha\in J\}$ restricted to the complement of any neighborhood of the origin is tight and uniformly bounded in total variation norm. Clearly, the same applies to
         \begin{align*}
  \left\{\sum_{i=1}^n t_i \nu_{\alpha_i}:\, \alpha_i\in J,\  t_i\ge 0, \, t_1+\dots +t_n\le T \text{ for }i=1,\dots, n,\, n\in\N  \right\},
     \end{align*}
  and another application of Prohorov's Theorem shows that this set  restricted to the complement of any neighborhood of the origin is relatively compact.
\item[(3)] For every $n,\,N\in \N$ one obtains
\begin{align*}
&  \sup_{\substack{t_1+\dots +t_n\le T\\ t_1,\dots,t_n\ge 0}}\sup_{\alpha_1,\dots,\alpha_n\in J}   \sum_{k=N}^\infty\sum_{i=1}^n \scapro{t_i T_{\alpha_i}f_k}{f_k}\\
&\qquad \qquad\le  \sup_{\substack{t_1+\dots +t_n\le T\\ t_1,\dots,t_n\ge 0}}\sum_{i=1}^n t_{i} \;
\sup_{\alpha\in J}\sum_{k=N}^\infty\scapro{ T_{\alpha}f_k}{f_k}\\
&\qquad \qquad\le T \sup_{\alpha\in J}\sum_{k=N}^\infty\scapro{ T_{\alpha}f_k}{f_k}
\to 0 \qquad \text{as }N\to\infty.
\end{align*}
Analogously we conclude
    \begin{align*}
    \sup_{\substack{t_1+\dots +t_n\le T\\ t_1,\dots,t_n\ge 0}}
\sup_{\alpha_1,\dots,\alpha_n\in J}   \sum_{k=1}^\infty\sum_{i=1}^n \scapro{t_i T_{\alpha_i}f_k}{f_k}
\le    T \sup_{\alpha\in J}\sum_{k=1}^\infty\scapro{ T_{\alpha}f_k}{f_k}<\infty.
\end{align*}
\end{enumerate}
The proof is completed.
\end{proof}

For the proof of Theorem \ref{th.stoch-op-continuous} we now introduce an alternative definition of a stochastic integral $\tilde{I}(\Psi)(t)$ for simple integrands $\Psi$. Its definition guarantees that the integral processes $(\tilde{I}(\Psi_n)(t):\, t\in [0,T])$ for an approximating sequence $(\Psi_n)$ of simple integrands converge uniformly in probability (Proposition \ref{pro.tilde-I}), which guarantees that its limit has  c{\`a}dl{\`a}g trajectories. The original stochastic integral $I(\Psi_n)(t)$ converges only as a $V$-valued random variable, i.e.\ at each fixed time $t\in [0,T]$; see Theorem \ref{th.stoch-op-continuous}.

For the definition assume that $\Psi\in \H_0(U,V)$ is of the form
\begin{align*}
  \Psi(t)=\Phi_0\1_{\{0\}}(0)+ \sum_{j=1}^{N} \Phi_j \1_{(t_j,t_{j+1}]}(t)
  \qquad\text{for all }t\in [0,T],
\end{align*}
where $0=t_1< \cdots < t_{N+1}=T$ is a finite sequence of
deterministic times  and each $\Phi_j\colon \Omega\to \L_2(U,V)$ is an
$\F_{t_j}$-measurable random variable for each $j=0,\dots, N$.
Then we define a $V$-valued stochastic process $(\tilde{I}(\Psi)(t):\,
t\in [0,T])$ by
\begin{align*}
\tilde{I}(\Psi)(t):=\begin{cases}\displaystyle
 0, & \text{if }t\in [0,t_2),\\\displaystyle
 \sum_{j=1}^k \Phi_j \big(L(t_{j+1})-L(t_j)\big),&\text{if } t\in [t_{k+1},t_{k+2}) \text{ for }k\in \{1,\dots, N-1\},\\\displaystyle
  \sum_{j=1}^N \Phi_j \big(L(t_{j+1})-L(t_j)\big),&\text{if } t=T.
\end{cases}
\end{align*}
Obviously, $\tilde{I}(\Psi)$ is an adapted stochastic process in $V$ with
 c{\`a}dl{\`a}g trajectories.

\begin{proposition}\label{pro.tightness}
For every $n\in\N$, let $\Psi_n$ be a simple stochastic process in $\H_0(U,V)$ defined on a partition $\{t_{n,k}\}_{k=1\dots, N_n+1}$ and define $\G_n:=\{\F_{t_{n,k}}:\,k=1,\dots, N_n+1\}$. If $\{\Psi_n:\, n\in\N\}$ is tight in $\M_1\big(D_-\big([0,T];\L_2(U,V)\big)\big)$ then
\begin{align*}
  \big\{ \tilde{I}(\Psi_n)(\tau):\, \,\tau\in\Upsilon(\G_n),\, n\in\N\big\}
\end{align*}
is tight in $V$, where
$
\Upsilon(\G_n)=\big\{\tau\colon \Omega\to\{t_{n,1},\,\dots, t_{n,N_n+1}\}:\text{ is stopping
time for $\G_n$}.\big\}
$
\end{proposition}
\begin{proof}
 Each $\Psi_n$ is of the form
\begin{align*}
  \Psi_n(t)=\Phi_{n,0}\1_{\{0\}}(t)+\sum_{j=1}^{N_n} \Phi_{n,j} \1_{(t_{n,j},t_{n,j+1}]}(t),
\end{align*}
for $0=t_{n,1}< \cdots < t_{n,N_n+1}=T$   and $\Phi_{n,j}\in L^0_P(\Omega,\F_{t_{n,j}};\L_2)$ for each $j=0,\dots, N_n$ and $n\in\N$.
Define for each  $j=2,\dots, N_n+1$ and $n\in\N$ the $\F_{t_{n,j}}$-measurable random variable
\begin{align*}
  X_{n,j}:=\Phi_{n,j-1}\big(L( t_{n,j})-L(t_{n,j-1})\big)\colon\Omega\to V,
\end{align*}
and choose a regular conditional distribution
\begin{align*}
  P_{n,j}\colon \Borel(V)\times \Omega\to [0,1],\qquad
   P_{n,j}(B,\omega)=P\big(X_{n,j}\in B\,\vert \,\F_{t_{n,j-1}}\big)(\omega).
\end{align*}
Lemma \ref{le.conditional-exp} guarantees for every $v\in V$ that
\begin{align*}
  E\big[\exp(i \scapro{X_{n,j}}{v})\,\big\vert \,\F_{t_{n,j-1}}\big]
=\exp\big((t_{n,j}- t_{n,j-1})S(\Phi^\ast_{n,j-1} v)\big)\quad\text{$P$-a.s.,}
\end{align*}
where $S\colon U\to \C$ denotes the cylindrical L{\'e}vy symbol of $L$.
It follows for $P$-a.a. $\omega\in \Omega$ that
\begin{align}\label{eq.cond-dist-lambda}
  P_{n,j}(\cdot,\omega)=\Big(\lambda\circ \big(\Phi_{n,j-1}(\omega)\big)^{-1}\Big)^{\ast( t_{n,j}-t_{n,j-1})},
\end{align}
where $\lambda$ is the cylindrical distribution of $L(1)$. For $\tau\in \Upsilon(\G_n)$
introduce the notation
\begin{align*}
[\tau](\omega):=\inf\big\{k\in \{1,\dots, N_n+1\}:\, \tau(\omega)=t_{n,k}\big\}.
\end{align*}
 Define for
every stopping time $\tau\in \Upsilon(\G_n)$ with $2\le [\tau]\le N_n+1$ the random probability measure
\begin{align*}
P_n(\tau)\colon\Borel(V)\times \Omega\to [0,1],\qquad
 P_n(\tau)=P_{n,2}\ast \dots \ast P_{n,[\tau]}.
\end{align*}
Let $\epsilon>0$ be given. Since $\{\Psi_n:\, n\in\N\}$ is tight there exists a compact set $C\subseteq D_-\big([0,T];\L_2(U,V)\big)$ such that $P(\Psi_n\in C)\ge 1-\epsilon$ for all $n\in\N$. Proposition 1.6 in \cite{Jakubowski-OnSkorokhod} guarantees that there exists a compact set $K\subseteq \L_2(U,V)$ such that $\{\Psi_n\in C\}\subseteq \{\Psi_n(t)\in K\text{ for all } t\in [0,T]\}$ for all $n\in\N$.
Consequently, the set
\begin{align}\label{eq.compactL-2}
 A_n:=\{\Psi_n(t)\in K\text{ for all }t\in [0,T]\}
 =\{\Phi_{n,j}\in K \text{ for all  }j=0,\dots ,N_n\},
\end{align}
satisfies $P(A_n)\ge 1-\epsilon$ for all $n\in\N$.
Denoting $\lambda_{\phi}:=\lambda\circ\phi^{-1}$ for every $\phi\in K$,  Proposition \ref{pro.condcompunderK} guarantees that the set  $\{\lambda_{\phi}:\, \phi\in K\}$ of infinitely divisible probability measures $\lambda_{\phi}$ is relatively compact in $\M_1(V)$.  Lemma \ref{le.infdiv-convex}  yields that the set
\begin{align*}
{\mathfrak X}:=\Big\{\lambda_{\phi_1}^{\ast s_{1}}\ast \dots \ast \lambda_{\phi_{n}}^{\ast s_{n}}:
 s_j\ge 0,\, s_1+\dots +s_n\le T,\, \phi_j\in K,\, j=1,\dots, n,\, n\in\N\Big\}
\end{align*}
is relatively compact in $\M_1(V)$. Since \eqref{eq.cond-dist-lambda} implies
\begin{align*}
  &\big\{P_n(\tau(\omega))(\cdot,\omega):\,  \tau\in \Upsilon(\G_n),\, 2\le [\tau]\le N_n+1, \,\omega\in A_n,\, n\in\N \big\}\\
  &\qquad \subseteq \big\{P_n(k)(\cdot,\omega):\, k\in\{2,\dots, N_n+1\}, \,\omega\in A_n,\, n\in\N \big\}\\
  &\qquad   \subseteq {\mathfrak X},
\end{align*}
it follows that the set
\begin{align*}
  \{P_n(\tau):\, \tau\in \Upsilon(\G_n),\, 2\le [\tau]\le N_n+1,\, n\in\N \}
\end{align*}
of random probability measures is tight. Theorem \ref{th.tightness-tangent} implies that
\begin{align*}
  \{X_{n,2}+\dots + X_{n,[\tau]}:\,  \tau\in \Upsilon(\G_n),\, 2\le[\tau]\le N_n+1,\, n\in\N\}
\end{align*}
is tight which completes the proof.
\end{proof}

\begin{proposition}\label{pro.tilde-I}
For every sequence $(\Psi_n)_{n\in\N}\subseteq \H_0(U,V)$
which converges to some $\Psi\in \H(U,V)$ in probability in the
Skorokhod metric, i.e.\
\begin{align*}
\lim_{n\to\infty} P\left(d_J\big(\Psi_n,\Psi)\big)\ge \epsilon\right)=0
\qquad\text{for all $\epsilon>0$, }
\end{align*}
there exists a $V$-valued, adapted stochastic process $(I(\Psi)(t):\,t\in [0,T])$ with  c{\`a}dl{\`a}g trajectories obeying for each $\epsilon>0$:
  \begin{align*}
\lim_{n\to\infty} P\left(\sup_{t\in [0,T]}\norm{\tilde{I}(\Psi_n)(t)-I(\Psi)(t)}\ge \epsilon\right)=0.
  \end{align*}
 The limit $I(\Psi)$ does not depend on the sequence $(\Psi_n)_{n\in\N}$, i.e.\ it is unique up to evanescence.
\end{proposition}

\begin{proof}
It is suffficient
to show for an arbitrary $\epsilon>0$ that
\begin{align}\label{eq.claim1}
 \lim_{m,n\to\infty} P\left(\sup_{t\in [0,T]}\norm{\tilde{I}(\Psi_m)(t)-\tilde{I}(\Psi_n)(t)}> \epsilon\right)=0.
\end{align}
Recall that the $V$-valued stochastic process $(\tilde{I}(\Psi_m)(t):\, t\in [0,T])$ has  c{\`a}dl{\`a}g paths due to its definition. Define for each $m,n\in\N$ the stopping time
\begin{align*}
\tau_{m,n}:=\inf\left\{ t>0:\, \norm{\tilde{I}(\Psi_m)(t)-\tilde{I}(\Psi_n)(t)}>\epsilon\right\}\wedge T,
\end{align*}
where $\inf\{\emptyset\} =\infty$. By the very definition of $\tau_{m,n}$ it follows that
\eqref{eq.claim1} is satisfied if and only if
\begin{align}\label{eq.claim2}
   \lim_{m,n\to\infty} P\Big(\norm{\tilde{I}(\Psi_m)(\tau_{m,n})-\tilde{I}(\Psi_n)(\tau_{m,n})}> \epsilon\Big)=0.
\end{align}
In order to establish \eqref{eq.claim2}, it is according to Lemma 2.4 in \cite{Jakubowski88} sufficient to show
\begin{align}
 &\text{(i) } \left\{\tilde{I}(\Psi_m)(\tau_{m,n})-\tilde{I}(\Psi_n)(\tau_{m,n}):\, m,n\in \N\right\} \text{ is tight in $V$};\label{eq.set-tight}\\
 &\text{(ii) } \text{for every  $v\in V$ we have}:\notag\\
 &\qquad\qquad \lim_{m,n\to\infty}\big\langle\tilde{I}(\Psi_m)(\tau_{m,n})-\tilde{I}(\Psi_n)(\tau_{m,n})\big\rangle \big\langle v \big\rangle= 0
  \qquad \text{in probability}.\label{eq.conv-weak}
\end{align}
By merging the partitions where $\Psi_m$ and $\Psi_n$ are defined on
 we obtain for every  $m$, $n\in\N$ the representation
\begin{align*}
  \Psi_m(t)-\Psi_n(t)=\Phi_{m,n,0}\1_{\{0\}}(t)+\sum_{j=1}^{N_{m,n}} \Phi_{m,n,j} \1_{(t_{m,n,j},t_{mn,j+1}]}(t)\quad\text{for all }t\in [0,T],
\end{align*}
where $0=t_{m,n,1}<  \cdots < t_{m,n,N_{m,n}+1}=T$ is a finite sequence of
deterministic times  and $\Phi_{m,n,j}\colon\Omega\to\L_2(U,V)$  is an
$\F_{t_{m,n,j}}$-measurable  random variable   for each $j=0,\dots, N_{m,n}$.

In order to establish \eqref{eq.set-tight} note that for each $m$, $n\in\N$ the stochastic process $(\tilde{I}(\Psi_m-\Psi_n)(t):\, t\in [0,T])$
varies only at points of the partition $\pi_{m,n}:=\{t_{m,n,k}:\, k=2,\dots, N_{m,n,}+1\}$.
Consequently, the stopping time $\tau_{m,n}$ only attains values in $\pi_{m,n}$ and thus it follows  $\tau_{m,n}\in \Upsilon(\G_{m,n})$ for $\G_{m,n}
:=\{\F_{t_{m,n,k}}:k=1,\dots, N_{m,n}+1\}$.  Consequently, one can apply Proposition \ref{pro.tightness} to conclude \eqref{eq.set-tight}.

For establishing \eqref{eq.conv-weak},
 define for every $j=2,\dots, N_{m,n}+1$ the $V$-valued random variable
\begin{align*}
  X_{m,n,j}:=\Phi_{m,n,j-1}\big(L( t_{m,n,j})-L(t_{m,n,j-1})\big)\colon \Omega\to V.
\end{align*}
Obviously, we have
\begin{align}\label{eq.X-and-I}
\tilde{I}(\Psi_m)(\tau_{m,n})-\tilde{I}(\Psi_n)(\tau_{m,n})=
X_{m,n,2}+\dots + X_{m,n,[\tau_{m,n}]},
\end{align}
where we use the notation
\begin{align*}
[\tau_{m,n}](\omega):=\inf\big\{k\in \{2,\dots, N_{m,n}+1\}:\, \tau_{m,n}(\omega)=t_{m,n,k}\big\}.
\end{align*}
Lemma \ref{le.conditional-exp} implies  for all $\beta\in\R$ and $v\in V$:
\begin{align*}
F_{m,n,j}(\beta):&=E\left[ e^{i \beta \scapro{X_{m,n,j}(t)}{v}}|\F_{t_{m,n,j-1}}\right]\\
&=  \exp\left( \left(t_{m,n,j}- t_{m,n,j-1}\right) S \left(\beta \Phi_{m,n,j-1}^\ast  v \right)  \right)\quad\text{$P$-a.s.}
\end{align*}
Consequently, we obtain $P$-a.s. that
\begin{align}
  F_{m,n}(\beta):=\prod_{j=2}^{[\tau_{m,n}]} F_{m,n,j}(\beta)
 &=\exp\left( \sum_{j=2}^{[\tau_{m,n}]}\left(t_{m,n,j}- t_{m,n,j-1}\right) S \left( \beta\Phi_{m,n,j-1}^\ast  v\right)  \right)\notag\\
  &= \exp\left(\int_0^{\tau_{m,n}}  S\big(\beta (\Psi_{m}^\ast(s)-\Psi_n^\ast(s)) v \big)\,ds\right).\label{eq.prod-cond-prob}
\end{align}
In order to show $F_{m,n}(\beta)\to 1$ in probability for $m,n\to\infty$ we have to show that each subsequence $(F_{m_k,n_k}(\beta))_{k\in\N}$ has a further subsequence converging  to $1$ $P$-a.s. As $d_J(\Psi_{m_k},\Psi)$ and $d_J(\Psi_{n_k},\Psi)$ converge to $0$ in probability
for $k\to\infty$ there exists subsequences  $(d_J(\Psi_{m_{k_\ell}},\Psi))_{\ell\in\N}$ and $(d_J(\Psi_{n_{k_\ell}}^\prime,\Psi))_{\ell\in\N}$ converging to $0$ $P$-a.s.\ for $\ell\to\infty$. It follows that there exists a set $\Omega_0\in \A$ with $P(\Omega_0)=1$ such that for each
$\omega\in\Omega_0$ we have that
\begin{align}
&\sup_{\ell\in\N}\sup_{t\in [0,T]} \norm{\Psi_{m_{k_\ell}}(t)(\omega)-\Psi_{n_{k_\ell}}(t)(\omega)}_{\L_2}<\infty, \label{eq.diff-bounded}
\intertext{
and that there exists a  Lebesgue null set $N_\omega\subseteq [0,T]$ depending on $\omega$ such that}
&
  \lim_{\ell\to\infty}\norm{\Psi_{m_{k_\ell}}(s)(\omega)-\Psi_{n_{k_\ell}}(s)(\omega)}_{\L_2}=0
  \quad\text{for all }s\in [0,T]\setminus N_\omega.\label{eq.diff-almost-null}
\end{align}
Lemma 3.2 in \cite{Riedle-Cauchy} guarantees that the cylindrical L{\'e}vy symbol $S$ maps bounded sets to bounded sets. Consequently, we can conclude from Lebesgue's Theorem of dominated convergence by applying \eqref{eq.diff-bounded} and \eqref{eq.diff-almost-null},  that
\begin{align*}
 & \lim_{\ell\to\infty} \abs{\int_0^{\tau_{{m_{k_\ell}},{n_{k_\ell}}}(\omega)}  S\big(\beta (\Psi_{{m_{k_\ell}}}^\ast(s)(\omega)-\Psi_{n_{k_\ell}}^\ast(s)(\omega)) v \big)\,ds}\\
 &\qquad \qquad \le   \lim_{\ell\to\infty} \int_0^T  \abs{S\big(\beta (\Psi_{{m_{k_\ell}}}^\ast(s)(\omega)-\Psi_{n_{k_\ell}}^\ast(s)(\omega)) v \big)}\,ds =0.
\end{align*}
Since we considered an arbitrary subsequence we can conclude that  $F_{m,n}(\beta)\to 1$ in probability for $m,n\to\infty$ for every $\beta\in\R$.   The principle of conditioning, Theorem \ref{th.conditioning},  yields
\begin{align*}
\lim_{m,n\to\infty} E\left[ \exp\left(i \beta \sum_{j=2}^{[\tau_{m,n}]} \scapro{X_{m,n,j}}{v}\right)\right]
=1\qquad\text{for every }\beta\in\R.
\end{align*}
Because of the representation \eqref{eq.X-and-I} this establishes \eqref{eq.conv-weak}.

Let $(\Psi_n)_{n\in\N}$ and $(\Psi_n^\prime)_{n\in\N}$ be two sequences  converging to $\Psi$ in probability in the Skorokhod metric $d_J$
and denote by $I(\Psi)$ and $I^\prime(\Psi)$ the limits of
$(\tilde{I}(\Psi_n))_{n\in\N}$ and $(\tilde{I}(\Psi_n^\prime))_{n\in\N}$. As in the proof of \eqref{eq.conv-weak} we can conclude that
\begin{align*}
  \lim_{n\to\infty}\scapro{\tilde{I}(\Psi_n-\Psi_n^\prime)(t)}{v}
  =0 \quad\text{in probability for all $v\in V$ and $t\in [0,T]$},
\end{align*}
which shows $I(\Psi)(t)=I^\prime(\Psi)(t)$ P-a.s.\ for each $t\in [0,T]$.
Since the stochastic processes $I(\Psi)$ and $I^\prime(\Psi)$ have  c{\`a}dl{\`a}g paths it follows that they are indistinguishable.
\end{proof}

Combining Lemma \ref{le.approximation} and Proposition \ref{pro.tilde-I} enables us to define
for every $\Psi\in\H(U,V)$:
\begin{align*}
\I(\Psi):=\lim_{n\to\infty}\widetilde{I}(\Psi_n),
\end{align*}
where $(\Psi_n)_{n\in\N}$ is an approximating sequence of simple processes in $\H_0(U,V)$
and the limit is in probability in the uniform norm as stated in Proposition \ref{pro.tilde-I}.
\begin{proposition}\label{pro.I-semi-martingale}
If $\Psi$ is in $\H(U,V)$ then the stochastic process $(I(\Psi)(t):\, t\in [0,T])$
is a semi-martingale.
\end{proposition}
\begin{proof}
Denote by $\E(V,V)$ the space of all adapted,  c{\`a}gl{\`a}d, simple processes with values in $\L(V,V)$ which are bounded by $1$, that is each $\Theta\in \E(V,V)$ is of the form
\begin{align}\label{eq.defsimple-E}
  \Theta(t)=\Gamma_0\1_{\{0\}}(0)+ \sum_{k=1}^{N} \Gamma_j \1_{(s_j,s_{j+1}]}(t)
  \qquad\text{for all }t\in [0,T],
\end{align}
where $0=s_1< \cdots < s_{N+1}=T$ is a finite sequence of
deterministic times  and each $\Gamma_k\colon \Omega\to \L(V,V)$ is an
$\F_{s_k}$-measurable random variable with $\norm{\Gamma_{k}(\omega)}_{V\to V}\le 1$
for all $\omega\in \Omega$ and for each $k=0,\dots, N$.
The elementary integral is then defined by
\begin{align*}
\int_0^T \Theta(s)\,\, I(\Psi)(ds)
= \sum_{k=1}^N \Gamma_k \big(\I(\Psi)(s_{k+1})-I(\Psi)(s_k)\big).
\end{align*}
Theorem 2.1 in \cite{Jakubowskietal} shows that $I(\Psi)$ is a semi-martingale if and
only
\begin{align}\label{eq.Psi-bounded-prob}
\left\{ \norm{ \int_0^T \Theta(s)\, I(\Psi)(ds)}:\, \Theta\in \E(V,V)\right\}\quad
\text{ is stochastically bounded.}
\end{align}
Lemma \ref{le.approximation} guarantees that there exists a sequence $(\Psi_n)_{n\in\N}$
of simple processes in $\H_0(U,V)$ converging to $\Psi$ in probability in the Skorokhod metric.
As  Proposition \ref{pro.tilde-I} implies
that $\tilde{I}(\Psi_n)(s)$ converges to $I(\Psi)(s)$ in probability for all $s\in [0,T]$, it follows that  \eqref{eq.Psi-bounded-prob} is established by showing
\begin{align}\label{eq.Psin-bounded-prob}
\left\{ \norm{ \int_0^T \Theta(s)\,\, \tilde{I}(\Psi_n)(ds)}: \Theta\in\E(V,V),\, n\in\N\right\}\quad
\text{ is stochastically bounded.}
\end{align}
Each $\Psi_n$ is of the form
\begin{align}\label{eq.approx-Psin}
  \Psi_n(t)=\Phi_{n,0}\1_{\{0\}}(t)+\sum_{\ell=1}^{N_n} \Phi_{n,\ell} \1_{(t_{n,\ell},t_{n,\ell+1}]}(t)\qquad\text{for all }t\in [0,T],
\end{align}
for $0=t_{n,1}< \cdots < t_{n,N_n+1}=T$   and $\Phi_{n,\ell}\in L^0_P(\Omega,\F_{t_{n,\ell}};\L_2)$ for each $\ell=1,\dots, N_n$ and $n\in\N$.
For given $\Theta\in \E(V,V)$ of the form \eqref{eq.defsimple-E} and $\Psi_n\in \H_0(U,V)$
of the form \eqref{eq.approx-Psin} we can assume, by possibly enlarging the partition $(t_{n,\ell})_{\ell=1,\dots, N_n+1}$, that for every $k\in\{1,\dots, N+1\}$ there exists $\ell_{k}\in \{1,\dots, N_n+1\}$ such that $s_k=t_{n,\ell_{k}}$. It follows that
\begin{align*}
\int_0^T \Theta(s)\,\, \tilde{I}(\Psi_n)(ds)
&= \sum_{k=1}^N \Gamma_k \big( \tilde{I}(\Psi_n)(s_{k+1})-\tilde{I}(\Psi_n)(s_k)\big)\\
&= \sum_{k=1}^N \Gamma_k \left( \sum_{\ell=\ell_k}^{\ell_{k+1}-1} \Phi_{n,\ell}
 \big(L(t_{n,\ell+1}))-L(t_{n,\ell})\big)  \right)\\
&= \sum_{\ell=1}^{N_n}  \big(\widetilde{\Gamma}_{\ell}\circ \Phi_{n,\ell}\big)
  \big(L(t_{n,\ell+1}))-L(t_{n,\ell})\big),
\end{align*}
where we use the definition
\begin{align*}
\widetilde{\Gamma}_\ell:=
\Gamma_k\quad\text{for every } \ell\in \{\ell_k,\dots,\ell_{k+1}-1\}.
\end{align*}
Note, that $\widetilde{\Gamma}_\ell$ is $\F_{t_{n,\ell}}$-measurable for all $\ell=1,\dots, N_n$.
Define for each $n\in\N$ and $\ell=2,\dots, N_n+1$ the $\F_{t_{n,\ell}}$-measurable random variable
\begin{align*}
  X_{n,\ell}^{\Theta}:=\big(\widetilde{\Gamma}_{\ell-1}\circ\Phi_{n,\ell-1}\big)\big(L( t_{n,\ell})-L(t_{n,\ell-1})\big)\colon\Omega\to V.
\end{align*}
Obviously, we have
\begin{align*}
\int_0^T \Theta(s)\,\, I(\Psi_n)(ds)
=X_{n,2}^\Theta + \dots + X_{n,N_n+1}^\Theta.
\end{align*}
Choose a regular conditional distribution
\begin{align*}
  P_{n,\ell}^\Theta\colon \Borel(V)\times \Omega\to [0,1],\qquad
   P_{n,\ell}^\Theta(B,\omega)=P\big(X_{n,\ell}^\Theta\in B\,\vert \,\F_{t_{n,\ell-1}}\big)(\omega).
\end{align*}
Lemma \ref{le.conditional-exp} guarantees for every $v\in V$ that
\begin{align*}
  E\big[\exp(i \scapro{X_{n,\ell}^\Theta}{v})\,\big\vert \,\F_{t_{n,\ell-1}}\big]
=\exp\Big((t_{n,\ell}- t_{n,\ell-1})S\big((\widetilde{\Gamma}_{\ell-1}\circ \Phi_{n,\ell-1})^\ast v\big)\Big)\quad\text{$P$-a.s.,}
\end{align*}
where $S\colon U\to \C$ denotes the cylindrical L{\'e}vy symbol of $L$.
It follows for $P$-a.a. $\omega\in \Omega$
\begin{align}\label{eq.P-Theta}
  P_{n,\ell}^\Theta(\cdot,\omega)=\Big(\lambda\circ \big((\widetilde{\Gamma}_{\ell-1}(\omega)\circ \Phi_{n,\ell-1}(\omega)\big)^{-1}\Big)^{\ast( t_{n,\ell}-t_{n,\ell-1})},
\end{align}
where $\lambda$ is the cylindrical distribution of $L(1)$.
 Define for
every $k\in\{2,\dots,  N_n+1\}$ the random probability measure
\begin{align*}
P_n^\Theta(k)\colon\Borel(V)\times \Omega\to [0,1],\qquad
 P_n^\Theta(k)=P_{n,2}^\Theta\ast \dots \ast P_{n,k}^\Theta.
\end{align*}
Let $\epsilon>0$ be given. Since $\{\Psi_n:\, n\in\N\}$ is tight we can conclude
as in the proof of Proposition \ref{pro.tightness} by using
Proposition 1.6 in \cite{Jakubowski-OnSkorokhod} that there exists a compact set $K\subseteq \L_2(U,V)$ such that the sets
\begin{align*}
 A_n:&=\{\Phi_{n,\ell}\in K \text{ for all  }\ell=1,\dots ,N_n\},
\end{align*}
satisfy $P(A_n)\ge 1-\epsilon$ for all $n\in\N$. The ideal property of
$\L_2(U,V)$ guarantees that the set
\begin{align*}
K_{\E}:=\big\{ \theta\circ \psi:\, \psi\in K, \theta\in \L(V,V) \text{ with }
  \norm{\theta}_{V\to V}\le 1\big\}.
\end{align*}
is a subset of $\L_2(U,V)$. Moreover, as $K$ is compact it follows that $K_{\E}$ is bounded and satisfies \eqref{eq.HS-compact2}, and thus the  closure $\overline{K}_{\E}$ is a compact set in $\L_2(U,V)$.
Denoting $\lambda_{\sigma}:=\lambda\circ\sigma^{-1}$ for every $\sigma\in \overline{K}_{\E}$,  Proposition \ref{pro.condcompunderK} guarantees that the set  $\{\lambda_{\sigma}:\, \sigma\in \overline{K}_{\E}\}$ of infinitely divisible probability measures $\lambda_{\sigma}$ is relatively compact.  Lemma \ref{le.infdiv-convex}  yields that the set
\begin{align*}
{\mathfrak X}&:=\Big\{\lambda_{\sigma_1}^{\ast s_{1}}\ast \dots \ast \lambda_{\sigma_n}^{\ast s_{n}}:
 s_j\ge 0,\, s_1+\dots +s_n\le T,\,
 \sigma_j\in \overline{K}_{\E},\, j=1,\dots, n,\, n\in\N\Big\}
\end{align*}
is relatively compact. Since \eqref{eq.P-Theta} implies
\begin{align*}
  &\big\{P_n^\Theta(k))(\cdot,\omega):\,  k\in\{ 2, \dots, N_n+1\}, \,\omega\in A_n,\, \Theta\in \E(V,V),\, n\in\N \big\}\subseteq {\mathfrak X},
\end{align*}
it follows that the set
\begin{align*}
  \big\{P_n^\Theta(k):\, k\in\{2,\dots, N_{n}+1\},\,\Theta\in\E(V,V),\, n\in\N \big\}
\end{align*}
of random probability measures is tight. Theorem \ref{th.tightness-tangent} implies that
\begin{align*}
  \big\{X_{n,2}^\Theta+\dots + X_{n,k}^\Theta:\,  k\in\{2,\dots, N_{n}+1\},\,\Theta\in\E(V,V),\, n\in\N\big\}
\end{align*}
is tight which establishes \eqref{eq.Psin-bounded-prob}.

\end{proof}

\begin{proof} (Theorem \ref{th.stoch-op-continuous})\\
Let  $(\Psi_n)_{n\in\N}$ be the sequence of simple processes in $\H_0(U,V)$
converging to $\Psi$ in probability in the Skorokhod metric, which exists due
to Lemma \ref{le.approximation}. In particular, $\{\Psi_n(t):\, t\in [0,T]\}$
is in the closure of $\{\Psi(t):\,t\in [0,T]\}$ and
the partition $(t_{n,j})_{j=1,\dots, N_n+1}$ obeys
\begin{align}\label{eq.mash-zero}
\lim_{n\to\infty} \sup_{j=1,\dots, N_n}\abs{t_{n,j+1}-t_{n,j}}=0.
\end{align}
Proposition \ref{pro.tilde-I} and Proposition \ref{pro.I-semi-martingale}  guarantee the existence of the adapted semi-martingale  $(I(\Psi)(t):\, t\in [0,T])$ in $V$ obeying
\begin{align*}
\sup_{t\in [0,T]}\norm{\tilde{I}(\Psi_n)(t)-I(\Psi)(t)}
\to 0 \qquad\text{in probability.}
\end{align*}
It remains to show that for each $t\in [0,T]$ the  $V$-valued random variable
\begin{align*}
\Delta_n(t):= I(\Psi_n)(t)-\tilde{I}(\Psi_n)(t)\colon
\Omega\to V,
\end{align*}
converges to $0$ in probability for $n\to\infty$.
In order to show this, fix some $t\in (0,T)$  and  denote by $k_n$ the element in $\{1,\dots, N_n\}$ such that $t\in (t_{n,k_n},t_{n,k_n+1}]$  and by $\Phi_{n,k_n}$ the $\L_2(U,V)$-valued random variable satisfying $\Psi_n(t)=\Phi_{n,k_n}$ for all $n\in\N$. Thus, we obtain $\Delta_n(t)= \Phi_{n,k_n}\big(L(t)-L(t_{n,k_n})\big)$ for all $n\in\N$.

Choose  a regular conditional distribution
\begin{align*}
P_{n}\colon \Borel(V)\times \Omega\to [0,1],\qquad
P_n(B,\omega)=P(\Delta_n\in B | \F_{t_{n,k_n}})(\omega).
\end{align*}
As in the proof of Proposition \ref{pro.tightness} it follows that
$\{P_n:\, n\in\N\}$ is tight. By taking expectation we obtain that
$\{\Delta_n(t):\, n\in\N\}$ is tight.

Furthermore, Lemma \ref{le.conditional-exp} implies for every  $\beta\in\R$ and $v\in V$ that
\begin{align}\label{eq.conditional-Delta}
E\big[\exp(i\beta \scapro{\Delta_n(t)}{v})|\F_{n,k_n}\big]
&=\exp\big((t-t_{n,k_n})S(\beta\Phi_{n,k_n}^\ast v)\big)\qquad\text{$P$-a.s.,}
\end{align}
where $S\colon U\to\C$ denotes the cylindrical L{\'e}vy symbol of $L$.
The set $\{\Phi_{n,k_n}^\ast (\omega) v:\, n\in\N\}$ is uniformly bounded
as $\Phi_{n,k_n}(\omega)$ is in the closure of $\{\Psi(t)(\omega):\, t\in [0,T]\}$ for every  $\omega\in \Omega$ and the closure of the latter is compact by  Proposition 1.1 in \cite{Jakubowski-OnSkorokhod}.
As  $S$ maps bounded sets to bounded sets by Lemma 3.2 in \cite{Riedle-Cauchy}, we   conclude from \eqref{eq.mash-zero} and \eqref{eq.conditional-Delta}  that
\begin{align}
 \lim_{n\to\infty}
E\big[\exp(i\beta \scapro{\Delta_n(t)}{v})|\F_{n,k_n}\big]=0
 \quad\text{$P$-a.s. for every $\beta\in\R$.}
\end{align}
Taking expectation yields
$\scapro{\Delta_n(t)}{v}\to 0$
in probability for $n\to\infty$ for all $v\in V$.  Together with  tightness of
$\{\Delta_n(t):\, n\in\N\}$ it follows from Lemma 2.4 in \cite{Jakubowski88}  that
$\Delta_n(t)\to 0$ for $n\to\infty$ in probability, which yields
  \begin{align}\label{eq.conv-discretisation}
\lim_{n\to\infty} P\big(\norm{I(\Psi_n)(t)-I(\Psi)(t)}\ge \epsilon\big)=0 \qquad\text{for every }\epsilon>0.
  \end{align}

It remains to show that \eqref{eq.conv-discretisation} holds true for
each sequence $(\Psi^\prime_n)_{n\in\N}$  in $\H_0(U,V)$ converging to $\Psi$  in probability in the Skorokhod metric $d_J$; that is we have to
establish for every $t\in [0,T]$ that:
  \begin{align}\label{eq.conv-ast}
\lim_{n\to\infty} P\big(\norm{I(\Psi_n^\prime)(t)-I(\Psi)(t)}\ge \epsilon\big)=0
\qquad\text{for every }\epsilon>0.
  \end{align}
For this purpose define  for each $t\in [0,T]$ and $n\in\N$ the  $V$-valued random variable
\begin{align*}
\Delta_n^\prime(t):= I(\Psi_n)(t)-I(\Psi_n^\prime)(t)\colon
\Omega\to V,
\end{align*}
where $(\Psi_n)_{n\in\N}$ denotes the sequence from above.
Because of \eqref{eq.conv-discretisation}, we can establish \eqref{eq.conv-ast} by showing
\begin{align}\label{eq.conv-Delta-ast}
\lim_{n\to\infty} P(\norm{\Delta_n^\prime(t)}\ge\epsilon)=0
\qquad\text{for all }\epsilon>0.
\end{align}
In order to establish \eqref{eq.conv-Delta-ast}, it is according to Lemma 2.4 in \cite{Jakubowski88} sufficient to show:
\begin{align}
 &\text{(i) } \left\{ I(\Psi_n)(t)-I(\Psi_n^\prime)(t):\, n\in \N\right\} \text{ is tight in $V$};\label{eq.conv-Delta-set-tight}\\
 &\text{(ii) } \text{for every  $v\in V$ we have}:\notag\\
 &\qquad\qquad \lim_{n\to\infty}\big\langle I(\Psi_n)(t)-I(\Psi_n^\prime)(t)\big\rangle \big\langle v \big\rangle= 0
   \text{ in probability}.\label{eq.conv-Delta-weak}
\end{align}
By merging the partitions where $\Psi_n$ and $\Psi_n^\prime$ are defined  we obtain for every  $n\in\N$ the representation
\begin{align*}
  \Psi_n(t)-\Psi_n^\prime(t)=\Phi_{n,0}\1_{\{0\}}(t)+\sum_{j=1}^{N_{n}} \Phi_{n,j} \1_{(t_{n,j},t_{n,j+1}]}(t)\quad\text{for all }t\in [0,T],
\end{align*}
where $0=t_{n,1}<  \cdots < t_{n,N_{n}+1}=T$ is a finite sequence of
deterministic times  and $\Phi_{n,j}\colon\Omega\to\L_2(U,V)$  is an
$\F_{t_{n,j}}$-measurable  random variable   for each $j=0,\dots, N_{n}$.
For a fixed $t\in (0,T]$ we can assume that  for every $n\in\N$ there exists $k_n\in\{2,\dots, N_n+1\}$ such that $t=t_{n,k_n}$. Now we can prove
\eqref{eq.conv-Delta-set-tight} and \eqref{eq.conv-Delta-weak}
as \eqref{eq.set-tight} and  \eqref{eq.conv-weak} in the proof of Proposition \ref{pro.tilde-I}.
\end{proof}



\end{document}